\documentclass[11pt, a4paper,final]{amsart}
\usepackage[utf8]{inputenc}
\usepackage[english]{babel}   
\usepackage{latexsym, amsfonts, amsbsy,amsthm, amsmath, mathrsfs, amssymb}
\usepackage{enumerate}
\usepackage{amscd} 
\usepackage{epsfig} 
\usepackage{color}
\usepackage{url}
\usepackage[none]{hyphenat} 
\usepackage[all, knot]{xy} 
        \xyoption{arc} 
\usepackage{anysize}
\marginsize{2cm}{2cm}{2cm}{2cm}
\usepackage{hyperref}
\hypersetup{colorlinks,citecolor= blue,urlcolor=blue}
\usepackage{xr}

\usepackage{graphicx}
\usepackage{pstricks}
\usepackage{verbatim}
\usepackage[T1]{fontenc}
\usepackage{lmodern} 

\DeclareMathOperator{\tp}{tp}
\DeclareMathOperator{\qf}{qf}
\DeclareMathOperator{\dpr}{dpr}
\DeclareMathOperator{\bdn}{bdn}

\begin{document}


\def\Ind#1#2{#1\setbox0=\hbox{$#1x$}\kern\wd0\hbox to 0pt{\hss$#1\mid$\hss}
\lower.9\ht0\hbox to 0pt{\hss$#1\smile$\hss}\kern\wd0}
\def\ind{\mathop{\mathpalette\Ind{}}}
\def\Notind#1#2{#1\setbox0=\hbox{$#1x$}\kern\wd0\hbox to 0pt{\mathchardef
\nn=12854\hss$#1\nn$\kern1.4\wd0\hss}\hbox to
0pt{\hss$#1\mid$\hss}\lower.9\ht0 \hbox to
0pt{\hss$#1\smile$\hss}\kern\wd0}
\def\nind{\mathop{\mathpalette\Notind{}}}

\def\Cc{\mathfrak{C}}
\def\Ll{\mathcal{L}}
\def\Kk{\mathcal{K}}

\newcommand{\Set}[2]{\{  #1 \ | \ #2  \}}

\newtheorem{que}{Question}
\newtheorem*{que*}{Question}

\newtheorem{defi}{Definition}[section]
\newtheorem{exe}{Exercise}[section]
\newtheorem{lem}[defi]{Lemma}
\newtheorem{theo}[defi]{Theorem}
\newtheorem*{theo*}{Theorem}
\newtheorem{cor}[defi]{Corollary}
\newtheorem{que2}[defi]{Question}
\newtheorem{prop}[defi]{Proposition}
\newtheorem*{prop*}{Proposition}
\newtheorem{conj}{Conjecture}
\newtheorem{fact}[defi]{Fact}
\theoremstyle{definition}
\newtheorem{rk}[defi]{Remark}
\newtheorem{exa}[defi]{Example}
\newtheorem*{cla}{Claim}
\newenvironment{sol}{$\Box$}{\rule{1ex}{1ex}}
\newenvironment{sk}{Sketch of proof:}{\rule{1ex}{1ex}}
\newenvironment{clmproof}[1][\proofname]{\proof[#1]\renewcommand{\qedsymbol}{$\square$(claim)}}{\endproof}


\title{Non-forking and preservation of NIP and dp-rank}
\author{Pedro Andr\'es Estevan \and Itay Kaplan}

\thanks{The collaboration started during a research stay of the first author in The Hebrew University of Jerusalem partially supported by  Fundación Montcelimar, MTM2017-86777-P and 2017SGR-270.} 
\thanks{The second author would like to thank the Israel Science foundation for their support of this research (Grants no. 1533/14 and 1254/18).}

\begin{abstract}
        We investigate the question of whether the restriction of an NIP type $p\in S(B)$ which does not fork over $A\subseteq B$ to $A$ is also NIP, and the analogous question for dp-rank. We show that if $B$ contains a Morley sequence $I$ generated by $p$ over $A$, then $p\restriction AI$ is NIP and similarly preserves the dp-rank. This yields positive answers for generically stable NIP types and the analogous case of stable types. With similar techniques we also provide a new more direct proof for the latter. Moreover, we introduce a general construction of ``trees whose open cones are models of some theory'' and in particular an inp-minimal theory DTR of dense trees with random graphs on open cones, which exemplifies a negative answer to the question.
\end{abstract}

\maketitle

\section{Introduction}

This paper contributes to the study of classification-theoretic properties of types in arbitrary theories. More specifically we are interested in the following question:

\begin{que2}\label{question} Let $p\in S(B)$ be an NIP type which does not fork over $A\subseteq B$, is $p\restriction A$ NIP? \end{que2}

This type of question was first considered in \cite[Question 1.2]{HassonOnshuus} where ``NIP'' is replaced by ``stable''. There, the authors answered that question in the case of Rosy theories \cite[Corollary 2.27]{HassonOnshuus} (they also had an incorrect argument when the theory is NIP). In \cite[Corollary 2.5]{Gen} this (\textit{i.e.} the stable case) was positively resolved. A variant of Question \ref{question} was also posed for simple types in \cite[Problem 6.6]{Ch}. 


When the base set $A$ is a model and the type $p$ is definable over $A$, Question \ref{question} and its stable variant turn out to be quite easy (this in particular implies the stable variant of Question \ref{question} when the base is a model, which was already shown in \cite[Claim 2.19]{HassonOnshuus}).  

More is true: if $p$ is NIP/stable and does not co-divides over $A$ then $p\restriction A$ is NIP/stable (see Proposition \ref{pro:positive1} and Corollary \ref{cor:positive2}; see also Corollary \ref{cor:stable types nf over models}). The analogous result for dp-rank is also proved with co-forking instead of co-dividing (Proposition \ref{prop:preserving dp-rank under co-forking}) from which we easily deduce a preservation result for dp-rank under very strong assumptions (finite rank and additivity, see Corollary \ref{cor:if dp-rank is additive then preserved}). As we will see in Section \ref{sec:counterexample}, in its full generality Question \ref{question} has a negative answer, however we still have a partial positive answer. The following is Corollary \ref{cor:result2}.

\begin{theo*} Let $p\in S(B)$ be an NIP type which does not fork over $A\subseteq B$. If $I\subseteq B$ is an infinite Morley sequence generated by $p$ over $A$, then $p\restriction AI$ is NIP.\end{theo*} 

This result is a kind of approximation: if we can add a Morley sequence to the domain of the restriction, then we can ensure the NIP character of the type is preserved. Our proof adapts to show a stronger result: dp-rank is also preserved (See Theorem \ref{the:preserving dp-rank}).

\begin{theo*} Let $p\in S(B)$ be a type which does not fork over $A\subseteq B$. If there is a Morley sequence $I\subseteq B$ generated by $p$ over $A$, then $\dpr(p\restriction AI) = \dpr(p)$. \end{theo*}

In Section \ref{sec:applications} we give some applications. For instance, we show the equivalence of Lascar and Kim-Pillay splitting for NIP types (for types over models which are saturated enough); this is Corollary \ref{cor:Lascar=KP}. Also, from the theorems above we infer a sufficient condition for a positive answer to Question \ref{question} which also ensures preservation of dp-rank, namely that every Morley sequence generated by some global non-forking extension of $p$ over $A$ is totally indiscernible; this is Corollary \ref{cor:niptotind}. As a particular consequence we have a positive answer and preservation of dp-rank for global generically stable NIP types in Corollary \ref{cor:generically stable}. Using a variant of the proof of the theorems above and some extra argument we can recover the stable case (\textit{i.e.}, the main result of \cite{Gen}); this is Corollary \ref{cor:stable case, first proof}. Finally, with analogous techniques we give another more direct proof for the stable case (see Theorem \ref{the:resultstab} and Corollary \ref{cor:resultstab2}), without making use of generically stable types, which is a key tool in \cite{Gen}. All of this is done in Section \ref{sec:results}.

In order to generate a counterexample to Question \ref{question}, we introduce in Section \ref{sec:a general construction} a general method of constructing the theory $T^*$ of meet-trees whose open cones are models of a fixed theory $T$ with quantifier-elimination in a finite relational language. In fact, $T^*$ is NIP iff $T$ is (Corollary \ref{cor:T* NIP}). In Section \ref{sec: the type}, we define for every model $M \models T^*$ a partial type $\pi_T(x)$ over $M$ and one extra point which is finitely satisfiable in $M$, is NIP and moreover dp-minimal, distal and co-distal (see Definition \ref{def:distal and co-distal}), but its restriction to $M$ could have IP or positive dp-rank depending on the theory $T$. This construction shows that the addition of Morley sequences in the theorems above was necessary by taking $T$ to be the theory of the random graph (in which case $T^*$ is dubbed \emph{DTR}) or the theory of two independent equivalence relations (which shows this addition was necessary even for NIP theories). It also shows that restricting our attention to dp-minimal or distal types cannot help. 

If $T$ is a theory where there is a counterexample to Question \ref{question}, then $T$ must have SOP and IP (see Remark \ref{rem:nsop}). In particular $T$ must have TP$_1=\,$ SOP$_2$. Thus in terms of classification theory, the best we could hope for is an NTP$_2$ theory. Among those, inp-minimal, \textit{i.e.} NTP$_2$ of burden $1$, are tamest. Indeed in Section \ref{sec:classifying DTR} we prove that DTR is inp-minimal. 

In section \ref{sec:non-distal example}, we build a variant of DTR showing that there is a counterexample for Question \ref{question} also when the type is non-distal. 

We end this paper with some open questions, see Section \ref{sec:open questions}. 

The work in this paper involves concepts such as forking, Lascar-strong types, (Lascar-)invariance, products of types,... and techniques such as extraction of indiscernible (using Ramsey and compactness, \textit{i.e.} the so-called Standard Lemma \cite[Lemma 5.1.3]{TZ}), which we assume are well-known (but still we recall some of them). See e.g., \cite{Cas14} or \cite{Sim} for more details. We assume the reader is familiar with the standard conventions and notations of model theory. In particular, $T$ denotes a complete (first-order) theory in some language $\Ll$. $\mathfrak{C}$ denotes the monster model for $T$ (a large, saturated model). All sets $A$, $B$, ... and models $M$, $N$, ... considered are contained in the monster, and global types are types over $\mathfrak{C}$. We use Greek characters $\pi$, $\Sigma$, ... for partial types, $\varphi$, $\psi$, $\alpha$, $\beta$, ... for formulas and Latin characters $a$, $b$, $c$, ... for (possible infinite) tuples, we also write $a\in A$ when $a$ is a tuple with the obvious meaning. We write $ \models \varphi(a) $ to mean that $ \mathfrak{C} \models \varphi(a) $. For a property $t$, $\varphi^{(t)}$ means $ \varphi $ if $ t $ holds and $ \neg \varphi $ otherwise. In Section \ref{sec:counterexample} we use the symbol $\mathop{\&}$ for the logical conjunction because $\wedge$ is employed in the language of the structures considered there. 


\subsubsection*{Acknowledgments}

The authors would like to thank Enrique Casanovas for suggesting this question, Daniel Palac\'in for some conversations about the problem and Rosario Mennuni by pointing us Remark \ref{rem:rdos}. The first author would also like to thank Enrique Casanovas and Rafel Farr\`e for their comments and suggestions.

\section{Preliminaries}

\subsection{Stable types}


\begin{defi} A partial type $\pi(x)$ over $A$ is \emph{stable} if every extension $p\in S(B)$, over every $B\supseteq A$, is definable. \end{defi} 

Note that any extension of a stable type is stable. 

There are many equivalences of stability for types, see \textit{e.g.} \cite{Gen}, \cite[Section 10]{BCNS11}. We focus on those which are useful for our purposes. 
\begin{defi} We say that $\varphi(x,y)\in\mathcal{L}$ has \emph{OP} (the \emph{order property}) if there are sequences $(a_i)_{i<\omega}$ and $(b_i)_{i<\omega}$ such that $\models\varphi(a_i,b_j)^{(i<j)}$  for all $i,j<\omega$.

    A formula $\varphi(x,y)$ without the order property is called \emph{stable}.
    
    A formula $\varphi(x,y)$ has \emph{SOP} (the \emph{strict order property}) if it defines a partial order with infinite chains. A theory is \emph{NSOP} if no formula has SOP. 

 \end{defi}


The following characterization is well-known but we provide a proof.
\begin{lem} \phantomsection \label{lem:opform} \begin{enumerate}
\item $\varphi(x,y)$ has OP if and only if there is an indiscernible sequence $(a_i)_{i\in\mathbb{Z}}$ and a tuple $b$ such that $\models\varphi(a_i,b)^{(i\geq 0)}$ for all $i \in \mathbb{Z}$. 
\item $\varphi(x,y)$ has OP if and only if there is an indiscernible sequence $(b_i)_{i\in\mathbb{Z}}$ and a tuple $a$ such that $\models\varphi(a,b_i)^{(i\geq 0)}$ for all $i \in \mathbb{Z}$. 
\end{enumerate} \end{lem}
\begin{proof}  We show $(1)$. Assume we have $I=(a_i)_{i<\omega}$ and $J=(b_i)_{i<\omega}$ such that $\models\varphi(a_i,b_j)^{(i<j)}$ for all $i,j<\omega$. By Ramsey we may assume that $I$ and $J$ are indiscernible sequences. By compactness we may assume $I$ and $J$ indexed in $\mathbb{Z}$. Now take $a'_i=a_{-i}$ for every $i\in\mathbb{Z}$ and $b=b_1$. For the converse, note that for every $k\in\mathbb{Z}$, $(a_i)_{i\in\mathbb{Z}}\equiv (a_{i+k})_{i\in\mathbb{Z}}$. Applying an automorphism there is $c_k$ such that $(a_{i+k})_{i\in\mathbb{Z}} c_k\equiv (a_i)_{i\in\mathbb{Z}} b$ and so $\models\varphi(a_{i+k},c_k)^{(i\geq 0)}$, \textit{i.e.} $\models\varphi(a_i,c_k)^{(i\geq k)}$. Let $b_i=c_{-i+1}$ and $a'_i=a_{-i}$ for every $i<\omega$ and note that $\models\varphi(a'_i,b_j)^{(i<j)}$ for all $i,j<\omega$.\end{proof}

The following fact follows from the previous lemma (or its proof) in conjunction with the results in \cite[Section 10]{BCNS11}. 

\begin{fact}\label{fac:pstable} The following are equivalent for a partial type $\pi(x)$ over $A$:\begin{enumerate}
\item $\pi(x)$ is stable.
\item There are no sequences $(a_i)_{i<\omega}$, $(b_i)_{i<\omega}$, with $a_i\models\pi(x)$ and no formula $\varphi(x,y)\in\mathcal{L}$ such that $\models\varphi(a_i,b_j)^{(i<j)}$ for all $i,j<\omega$.
\item There is no $A$-indiscernible (or even $\emptyset$-indiscernible (see Remark \ref{rem:over emptyset})) sequence $(a_i)_{i\in\mathbb{Z}}$ of realizations of $\pi(x)$, no formula $\varphi(x,y)\in\mathcal{L}$ and no tuple $b$ such that $\models\varphi(a_i,b)^{(i\geq 0)}$ for all $i \in \mathbb{Z}$.
\item There is no $A$-indiscernible sequence $(a_i)_{i\in\mathbb{Z}}$, formula $\varphi(x,y)\in\mathcal{L}$ and no tuple $b\models\pi(x)$ such that $\models\varphi(b,a_i)^{(i\geq 0)}$ for all $i \in \mathbb{Z}$.
\end{enumerate} \end{fact}
Thus, if $\pi$ is unstable, \textit{i.e.} not stable, there is $\varphi(x,y)$ witnessing OP for $\pi$, meaning any of the equivalent points above.
\begin{rk} \label{rem:over emptyset}
    In Fact \ref{fac:pstable} (3), a general reason we could settle for $\emptyset$-indiscernibility is because given an $\emptyset$-indiscernible sequence $I$ as in there, we can extract from it (by Ramsey and compactness) an $A$-indiscernible sequence $I'$ so that $I \equiv I'$ and still all elements in $I'$ realize $\pi$, and applying an automorphism we can get some tuple $b'$ so that $b'I'\equiv bI$. 

    This asymmetry between (3) and (4) also appears in the NIP case, see Lemma \ref{lem:pnip} below. 
\end{rk} 
\begin{rk}\label{rem:pglobstable} It follows by compactness that the following statement also characterizes the stability of a partial type $\pi(x)$ closed under conjunctions: for each formula, there is another one in the type which prohibits OP. More precisely, the following are equivalent: \begin{enumerate}
\item $\pi$ is stable.  
\item For every $\varphi(x,y)\in\mathcal{L}$ there is $\alpha_{\varphi}(x,c)\in \pi$ such that $\varphi(x,y) \land \alpha_{\varphi}(x,c)$ is stable (as a formula with parameters) which (by Lemma \ref{lem:opform} and indiscernibility) is equivalent to saying that there is no $c$-indiscernible (or even $\emptyset$-indiscernible) sequence $(a_i)_{i\in\mathbb{Z}}$ and a tuple $b$ such that $a_i\models \alpha_{\varphi}(x,c)$ and $\models\varphi(a_i,b)^{(i\geq 0)}$ for all $i \in \mathbb{Z}$. 
\end{enumerate}
Note that by Lemma \ref{lem:opform}, $\varphi(x,y) \land \alpha_{\varphi}(x,c)$ is stable iff there is no $c$-indiscernible sequence $(b_i)_{i\in\mathbb{Z}}$ and no tuple $a$ such that $a\models \alpha_{\varphi}(x,c)$ and $\models\varphi(a,b_i)^{(i\geq 0)}$ for all $i \in \mathbb{Z}$.
\end{rk} 
\begin{rk} \label{rem:stable type restriction to a small set}
    From this remark it follows immediately that for any stable type $p \in S(A)$ there is some $B\subseteq A$ such that $p \restriction B$ is stable and $|B|\leq |T|$.
\end{rk}
\begin{fact}\label{fac:stableacl} \cite[Lemma 1.5]{Gen} Let $p\in S(A)$, then $p$ is stable if and only if any (some) extension to $acl^{eq}(A)$ is stable.\end{fact}

This fact also follows directly from the characterization in Fact \ref{fac:pstable} and the fact that in any theory, if $I$ is indiscernible over $A$ then it is also indiscernible over $acl(A)$. 

\subsection{NIP types}

Here we briefly go over the notions of NIP types and formulas. A general good reference for all things NIP is \cite{Sim}, but see also \cite{KSim} and \cite{BCNS11} for a  local approach (which is the one we take here).
   
\begin{defi} A formula $\varphi(x,y)\in\mathcal{L}$ has \emph{IP} (the \emph{independence property}) if there is a sequence $(a_i)_{i<\omega}$ such that for every $s\subseteq\omega$ the set of instances $\Set{\varphi(a_i,y)^{(i\in s)}}{i<\omega}$ is consistent.
 
We say that a partial type $\pi(x)$ \emph{has IP} if there is a formula $\varphi(x,y)\in\mathcal{L}$ which has IP as witnessed by a sequence $(a_i)_{i<\omega}$ of realizations $a_i\models\pi(x)$.
A formula or a partial type is \emph{NIP} if it does not have IP.  \end{defi} 
As in the stable case, note that any extension of an NIP type is itself NIP. Also, by Fact \ref{fac:pstable}, it follows immediately that any stable type is NIP. 

The following characterization of NIP formulas follows readily from \cite[Lemma 6.3]{Sim}. 
\begin{lem}\label{lem:ipform} Let $\varphi(x,y)\in\mathcal{L}$. Then $\varphi$ has IP if and only if one of the following equivalent conditions hold:\begin{enumerate}
\item There are sequences $(a_i)_{i<\omega}$, $(b_s)_{s\subseteq\omega}$ such that $\models\varphi(a_i,b_s)^{(i\in S)}$ for all $i<\omega$, $s\subseteq \omega$.
\item There are sequences $(a_s)_{s\subseteq\omega}$, $(b_i)_{i<\omega}$ such that $\models\varphi(a_s,b_i)^{(i\in S)}$ for all $i<\omega$, $s\subseteq \omega$.
\item There is an indiscernible sequence $(a_i)_{i<\omega}$ and a tuple $b$ such that $\models\varphi(a_i,b)^{(i\,\mbox{even})}$ for all $i<\omega$.
\item There is an indiscernible sequence $(b_i)_{i<\omega}$ and a tuple $a$ such that $\models\varphi(a,b_i)^{(i\,\mbox{even})}$ for all $i<\omega$.
\end{enumerate} \end{lem}

 
The next lemma follows from the definition of an NIP type and Lemma \ref{lem:ipform}, a proof can be found in \cite[Claim 2.1]{KSim} and \cite[Lemma 8.3 and Proposition 8.4]{BCNS11}.
\begin{lem}\label{lem:pnip} The following are equivalent for a partial type $\pi(x)$ over $A$:\begin{enumerate}
\item $\pi(x)$ is NIP: there are no sequences $(a_i)_{i<\omega}$, $(b_s)_{s\subseteq\omega}$, with $a_i\models\pi(x)$ and no formula $\varphi(x,y)\in\mathcal{L}$ such that $\models\varphi(a_i,b_s)^{(i\in S)}$ (this is the definition).
\item There are no sequences $(a_i)_{i<\omega}$, $(b_s)_{s\subseteq\omega}$, with $b_s\models\pi(x)$ and no formula $\varphi(x,y)\in\mathcal{L}$ such that $\models\varphi(b_s,a_i)^{(i\in S)}$.
\item There are no $A$-indiscernible (or even $\emptyset$-indiscernible) sequence $(a_i)_{i<\omega}$ of realizations of $\pi(x)$, no tuple $b$ and no formula $\varphi(x,y)\in\mathcal{L}$ such that $\models\varphi(a_i,b)^{(i\,\mbox{even})}$  for all $i<\omega$. 
\item There are no $A$-indiscernible sequence $(a_i)_{i<\omega}$, no tuple $b\models\pi(x)$ and no formula $\varphi(x,y)\in\mathcal{L}(A)$ such that $\models\varphi(b,a_i)^{(i\,\mbox{even})}$ for all $i<\omega$.
\end{enumerate} \end{lem}
As in the stable case, if $\pi$ has IP then there is $\varphi(x,y)$ and some sequence witnessing IP for $\pi$, meaning any of the equivalent points above. 
Analogous to Remark \ref{rem:pglobstable} we have:
\begin{rk}\label{rem:pglobnip} By compactness, for a partial type $\pi(x)$ closed under conjunctions, the following are equivalent:\begin{enumerate}
\item $\pi$ is NIP.
\item For every $\varphi(x,y)\in\mathcal{L}$ there is $\alpha_{\varphi}(x,c)\in \pi$ such that $\varphi(x,y) \land \alpha_{\varphi}(x,c)$ is NIP (as a formula with parameters) which means (by Lemma \ref{lem:ipform} and indiscernibility) that there is no $c$-indiscernible (or even $\emptyset$-indiscernible) sequence $(a_i)_{i<\omega}$ and tuple $b$ such that  $a_i\models\alpha_{\varphi}(x,c)$ and $\models \varphi(a_i,b)$ iff $i$ is even. 
\end{enumerate}
Note that By Lemma \ref{lem:ipform}, $\varphi(x,y) \land \alpha_{\varphi}(x,c)$ is NIP iff there is no $c$-indiscernible sequence $(b_i)_{i<\omega}$ and no tuple $a$ with $a\models\alpha_{\varphi}(x,c)$ and $\models \varphi(a,b_i)$ iff $i$ is even.
\end{rk} 

The analogous to Remark \ref{rem:stable type restriction to a small set} holds, \textit{i.e.} it follows from the remark above that if $p \in S(A)$ is NIP, then some restriction of $p$ to a set of size $\leq |T|$ is already NIP.


\begin{fact}\label{fac:tupnip} \cite[Theorem 4.11]{KOU} Let $p\in S(A)$ be an NIP type, and for $n<\omega$, let $a_0,a_1,\dots a_{n-1}$ be such that $a_i\models p$ for every $i<n$. Then the type $\tp(a_0,a_1,\dots a_{n-1}/A)$ is NIP.\end{fact}

\begin{rk}\phantomsection\label{rem:stsop}   \begin{enumerate}
\item If $\pi(x)$ is an NIP unstable type, then there is a formula in $\mathcal{L}$ with SOP.
\item In an NSOP theory all NIP types are stable.
\end{enumerate} \end{rk}
\begin{proof} (1) Since $\pi$ is unestable, there are indiscernible sequences $(a_i)_{i<\omega}$, $(b_N)_{N<\omega}$, with $a_i\models\pi$ and a formula $\varphi(x,y)$ such that $\models\varphi(a_i,b_N)^{(i<N)}$. Since $\pi$ is NIP, by definition, for $\varphi(x,y)$ and $(a_i)_{i<\omega}$ there is some $I\subseteq\omega$ such that $\Set{\varphi(a_i,y)^{(i\in I)}}{i<\omega}$ is inconsistent. Thus by compactness we have $n<\omega$ and a map $\eta$ as in the proof of \cite[Theorem 2.67]{Sim} and therefore we can apply the same proof. (2) follows directly from (1) by definition.\end{proof}

\subsection{Dp-rank}
\begin{defi} \label{def:dp-rank}
    Given a partial type $\Sigma(x)$ over $A$ and a cardinal $\kappa$, the \emph{dp-rank} of $\Sigma$ is less than $\kappa$, denoted by $\dpr(\Sigma)<\kappa$, if for every sequence of $A$-mutually indiscernible infinite sequences $(I_{\alpha})_{\alpha<\kappa}$, and for every $b\models\Sigma$ there is $\beta<\kappa$ such that $I_\beta$ is indiscernible over $bA$. We say $\dpr(\Sigma)\geq\kappa$ if the negation holds. We say that $\dpr(\Sigma) = \dpr(\Sigma')$ if for every cardinal $\kappa$, $\dpr(\Sigma)<\kappa$ iff $\dpr(\Sigma')<\kappa$ (and similarly we define $\dpr(\Sigma) \leq \dpr(\Sigma')$). If $\dpr(T) = \dpr(x=x) < 2$, we say that $T$ is \emph{dp-minimal}.
\end{defi}
See also \cite[Definition 1.1]{KSim} for a precise definition of dp-rank. 

\begin{rk} $A$ above can be replaced by any set containing the parameters of $\Sigma$ by \cite[Lemma 4.14]{Sim}. Note also that if $\Sigma' \supseteq \Sigma$ then $\dpr(\Sigma') \leq \dpr(\Sigma)$. \end{rk}
    
\begin{rk}\label{rem:ict pattern} \cite[Proposition 4.22]{Sim} For every partial type $\Sigma(x)$, $\dpr(\Sigma) \geq \kappa$ iff there is an \emph{ict-pattern} of depth $\kappa$, meaning a sequence of formulas $(\varphi_\alpha(x,y_\alpha))_{\alpha<\kappa}$, an array of tuples $(a_{i,\alpha})_{i<\omega,\alpha<\kappa}$ such that for any function $\eta:\kappa \to \omega$ there is some $b_\eta \models \Sigma$ such that $\varphi(b_\eta,a_{i,\alpha})$ holds iff $\eta(\alpha) = i$.  \end{rk}

\begin{rk} \label{rem:dp-rank infinite -> IP} \cite[Claim 2.1]{KSim} $\dpr(\Sigma) = \infty$ (i.e. $\geq \kappa$ for every cardinal $\kappa$) iff $\Sigma$ has IP.\end{rk}
    
We will use the following result which allows us to assume that the sequences witnessing the dp-rank are composed of realizations of the type (at the price of increasing the base). This was first proved in \cite{KaplanSimon} assuming NTP$_2$, and later proved in full generality by Chernikov in \cite[Theorem 5.3]{Ch}: 

\begin{fact} \label{fac:dp rank witnessing} Suppose that $\Sigma(x)$ is a partial type over $A$ with $\dpr(\Sigma(x)) \geq \kappa$. Then there is $A'$ containing $A$ such that $|A'| \leq \kappa + |A|$, a sequence $(I_\alpha)_{\alpha<\kappa}$ of infinite $A'$-mutually indiscernible infinite sequences such that for every element $b$ from any $I_\alpha$, $b \models \Sigma$, and some $d \models \Sigma$ such that for all $\alpha < \kappa$, $I_\alpha$ is not $A'd$ indiscernible. \end{fact}


\subsection{NIP, Forking and Lascar-invariance}

In this section we discuss forking, invariance and Lascar-invariance in the context of NIP types. For the definitions, see \cite{Cas14,Sim}. Some of the proofs here are based on Sections 8, 9 and 10 from \cite{BCNS11}.

Recall that a model $M$ is\emph{ $\omega$-saturated over $A\subseteq M$} iff for every $B\subseteq M$, with $|B|<\omega$ every finitary type over $AB$ is realized in $M$. A set $B$ is called \emph{complete over $A\subseteq B$} if every finitary type over $A$ is realized in $B$. The next remark will not be used, but it is good to recall some properties of non-splitting.
\begin{rk} Given a sufficiently saturated model $N$, if $p\in S(N)$ does not split over $A$ then it does not fork over $A$ (it is enough that $N$ is $\omega$-saturated over $A$, see Proposition 1 in \cite{Cas14}), moreover $p$ has a unique non-splitting over $A$ extension to every $C\supseteq N$ (it is enough $N$ complete over $A$, see Proposition 4 in \cite{Cas14}).\end{rk} 

Recall that a set $B$ is Lascar-complete over $A\subseteq B$ if every finitary Lascar-strong type is realized in $B$, \textit{i.e.} for every finite tuple $a$ there is $a'\in B$ with $a\equiv^{Ls}_A a'$.  
We leave the following remark as an exercise to the reader.
\begin{rk} \label{rem:saturated implies Lascar-complete} Given $N\supset A$, where $N$ is a $(|A|+|T|)^+$-saturated model.\begin{enumerate}
\item $N$ is $\omega$-saturated over $A$.
\item $N$ is Lascar-complete over $A$. 
\end{enumerate} \end{rk}

By \cite[Proposition 5.21]{Sim}, in an NIP theory non-forking is equivalent to Lascar-invariance. The next lemma proves the analogous results for NIP types. (This was claimed in Proposition 8.6 of \cite{BCNS11}, but the proof there was not entirely correct, as the indiscernibility should be over the parameters of the formula implying NIP.)


Recall that a global type is called \emph{invariant} over some small set $A$ if it does not split over $A$ (which is the same as being invariant under automorphisms fixing $A$). 

Similarly, a global type is \emph{Lascar-invariant} over $A$ if it does not Lascar-split over $A$ (\textit{i.e.} invariant under the group of $A$-Lascar-strong automorphisms: automorphisms fixing $\equiv^{Ls}_A$-classes).

\begin{lem} \phantomsection \label{lem:lascar} 
\begin{enumerate}
\item Let $p\in S(\mathfrak{C})$ be an NIP global type and $A\subseteq\mathfrak{C}$, then $p$ does not fork over $A$ if and only if $p$ is Lascar-invariant over $A$.
\item Let $p\in S(B)$ NIP non-forking over $A\subseteq B$, then $p$ does not Lascar-split over $A$.
\item Let $B$ be Lascar-complete over $A\subseteq B$. If $p(x)\in S(B)$ is Lascar non-splitting over $A$, then
for every $C\supseteq B$ there is a unique $p'\in S(C)$ Lascar non-splitting over $A$ such that $p\subset p'$.
\item Let $B\supseteq A$ be Lascar-complete over $A$. If $p\in S(B)$ is NIP, then $p$ does not fork over $A$ if and only if $p$ does not Lascar-split over $A$.
\item Thus, if $p$ is NIP and non-forking over a model $M$, it is non-Lascar-splitting over $M$, which is the same as non-splitting over $M$.
\item For an NIP global type $p$ we have the following equivalences: $p$ does not divide over $A$ iff it does not fork over $A$ iff it is Lascar-invariant over $A$.

When $p$ stable this is also equivalent to being $acl^{eq}(A)$-definable/invariant.

If additionally $A$ is a model, these conditions are also equivalent to being $A$-invariant in both the NIP and stable cases.
\end{enumerate}
\end{lem}
\begin{proof} (1) $(\Rightarrow)$ It is enough to show that for some $a,b$ starting an indiscernible sequence over $A$, $\varphi(x,a)\in p$ implies $\varphi(x,b)\in p$ for any $\varphi(x,y)\in\mathcal{L}(A)$. 
So assume towards a contradiction we have $a,b$ starting an indiscernible sequence $(a_i)_{i<\omega}$ over $A$ and $\varphi(x,a)\land\lnot\varphi(x,b)\in p$ for some $\varphi(x,y)\in\mathcal{L}(A)$. Since $p$ is NIP there is $\theta(x,c)\in p$ such that $\psi(x,y)=\theta(x,c)\land\varphi(x,y)$ is NIP. Let $I = (a_i)_{i \geq 2}$. Extracting, we can find $I'=a'_2, a'_3\dots$ indiscernible over $Aa_0a_1c$ such that $I\equiv_{Aa_0a_1c} I'$. In particular $I''=a_0,a_1,a'_2,a'_3\dots$ is $A$-indiscernible. Now let $c_0=c$ and for every $0<k<\omega$ let $c_{2k}$ be such that $a_0 a_1 a'_{\geq 2}c\equiv_A a'_{2k}a'_{2k+1} a'_{\geq 2k+2}c_{2k}$. Consider now the sequence $J=(b_k)_{k<\omega}$ where $b_k=a'_{2k}a'_{2k+1}c_{2k}$ for $0<k<\omega$ and $b_0=a_0a_1c$. Again extract an $A$-indiscernible sequence $\hat{J}=(\hat{b}_k)_{k<\omega}=(\hat{a}_{2k}\hat{a}_{2k+1}\hat{c}_{2k})_{k<\omega}$ with $EM(J/A)\subseteq EM(\hat{J}/A)$. We may assume $\hat{b}_0=b_0=a_0a_1c$. Consider $\chi(x,y_0,y_1,z)=\varphi(x,y_0)\land\lnot\varphi(x,y_1)\land\theta(x,z)$ and $\pi(x)=\Set{\chi(x,\hat{b_k})}{k<\omega}$. Now, since $\chi(x,b_0)\in p$ and $p$ does not fork over $A$ there is some $e\models\pi(x)$. Hence $\models\theta(e,c)$ and $\models\varphi(e,\hat{a}_i)^{(i\mbox{ even})}$ for all $i<\omega$, and in particular $\models\psi(e,\hat{a}_i)^{(i\mbox{ even})}$ for all $2 \leq i<\omega$. Since $(\hat{a}_i)_{2\leq i<\omega}$ is indiscernible over $Ac$ we have a contradiction with the choice of $\psi(x,y)$.

$(\Leftarrow)$  holds in general since for global types forking equals dividing over small sets and since elements in indiscernible sequences have the same Lascar-strong type.

(2) follows from $(1)$ taking a global non-forking extension. 

For (3) follow the proof of \cite[Proposition 4]{Cas14} replacing ``complete'' with ``Lascar-complete'' and ``splitting'' with ``Lascar splitting''. See also \cite{Cas14} after Remark 13.

(4) The direction which remained open now follows from $(3)$ and $(1)$.

For (5) and the last point of $(6)$ note that the notions of type and Lascar-strong type coincide over models. 

The only statement remaining is the point on stable types in (6). If $p$ does not fork over $A$ then it is invariant over any model containing $A$ (since stable types are NIP). Hence it is definable over any such model. For each formula $\varphi(x,y)$, let $e_\varphi$ be a canonical parameter of the formula defining $p\restriction \varphi$ in $\mathfrak{C}^{eq}$ (by this we mean that an automorphism fixes $p\restriction \varphi$ setwise if and only if it fixes $e_\varphi$). Hence $e_\varphi$ belongs to any model containing $A$ and thus to their intersection which is $acl^{eq}(A)$.

On the other hand, if $p$ is $acl^{eq}(A)$-invariant then it clearly is also $acl^{eq}(A)$-definable (since it is definable). 
\end{proof}


We also observe the following version of the Standard Lemma for extracting an indiscernible sequence while preserving the Lascar-strong type. 
\begin{lem}\label{lem:lascarseq} Let $I$ be an infinite indiscernible sequence over $A$ and $B\supseteq A$. Then there is some $B$-indiscernible sequence $J$ such that $J\equiv^{Ls}_A I$.

Moreover, let $I = (I_\alpha)_{\alpha<\kappa}$ be a sequence of $A$-mutually indiscernible sequences. Then there are some $B$-mutually indiscernible sequences $J = (J_\alpha)_{\alpha<\kappa}$ such that $J\equiv^{Ls}_A I$. 
\end{lem}
\begin{proof} For the first part, by Ramsey, compactness and applying an automorphism, there is some $M \supseteq A$ such that $I$ is indiscernible over $M$. Extract from $I$ an $MB$-indiscernible sequence $J$ of the same order type so that (as $I$ is already $M$-indiscernible), $J\equiv_M I$. In particular we have $J\equiv^{Ls}_A I$ and is $B$-indiscernible.

The ``moreover'' part is proved in the same way since the version of the Standard Lemma for mutually indiscernible sequences holds, see \textit{e.g.} \cite[Lemma 3.5]{HilsChernikov}, \cite[Lemma 4.2]{Sim}. 
\end{proof}


\subsection{Morley sequences and Generically Stable types}

\begin{defi} Suppose that $(X,<)$ is linearly ordered. \begin{enumerate}
\item A sequence $(a_i)_{i \in X}$ is \emph{generated by} $p\in S(B)$ over $A\subseteq B$ if $a_i\models p\restriction Aa_{<i}$ for all $i \in X$. 
\item A \emph{Morley sequence} over $A$ is a sequence $(a_i)_{i \in X}$ which is $A$-indiscernible and $A$-independent, \textit{i.e.} such that $\tp(a_i/Aa_{<i})$ does not fork over $A$ for every $i \in X$.
\item Given a type $p\in S(B)$ and $A\subseteq B$, \emph{a Morley sequence generated by} $p$ \emph{over} $A$ is a  sequence generated by $p$ over $A$ which is Morley over $A$.
\end{enumerate} \end{defi}

\begin{defi}
    If $q(x)$ and $r(y)$ are $A$-invariant global types, then the type $(q \otimes r)(x,y)$ is defined to be $\tp(a,b/\mathfrak{C})$ (in a bigger monster model) for any $b \models r$ and $a \models q\restriction \mathfrak{C}b$ (here we understand $q$ to mean its unique extension to a bigger model). (This can also be defined without stepping outside of the monster model, see \cite[Chapter 2]{Sim}.)
    
    For an ordinal $\alpha$, we define $q^{(\alpha)}(x_{\beta})_{\beta<\alpha}$ by induction:  $q^{(1)} = q$, $q^{(\alpha + 1)} = q(x_{\alpha}) \otimes q^{(\alpha)}(x_{<\alpha})$, taking unions at limits. 
 \end{defi}

 \begin{fact}\cite[Chapter 2]{Sim} \label{fac:sequences generated by invariant types}
    Given a global $A$-invariant type $q$ and an ordinal  $\alpha$, $q^{(\alpha)}$ is an $A$-invariant global type.  In addition, it is a type of an indiscernible sequence over $\mathfrak{C}$. 
    
    For any small set $B \supseteq A$, $q^{(\alpha)} \restriction B$ is given by $\tp(a_i)_{i<\alpha}$ where $(a_i)_{i<\alpha}$ is generated by $q$ over $B$.  Thus, this is a Morley sequence over both $B$ and $A$. 
\end{fact}

\begin{rk}\label{rem:exmor} 
    Suppose that $p$ is a global $A$-Lascar-invariant type. Hence $p$ is $M$-invariant for any model $M \supseteq A$, and hence by Fact \ref{fac:sequences generated by invariant types} so is $p^{(\alpha)}$ for any ordinal $\alpha$. It follows that $p^{(\alpha)}$ is also Lascar-invariant over $A$. 

    
    In this context, by saturation one can generate a Morley sequence of $p$ over any small model containing $A$. Such a sequence will be a Morley sequence over $A$.
    
    By Lemma \ref{lem:lascar} the above is true in particular when $p$ is NIP and does not fork over $A$.
\end{rk}




In the situation of Remark \ref{rem:exmor}, a sequence generated by $p$ over $A$ might not be indiscernible. The reason is that $p \restriction A$ does not determine a unique Lascar-strong type. This can be fixed if we already have an infinite indiscernible sequence generated by $p$:

\begin{lem}\label{lem:indlas2} Suppose that $p(x)\in S(\mathfrak{C})$ is a global $A$-Lascar-invariant type. Let $(X,<)$ be some infinite linearly ordered set. Let $I=(a_i)_{i\in X}$ be an $A$-indiscernible sequence such that $a_i\models p\restriction Aa_{<i}$ for all $i \in X$. 

Then, if $(Y,<)$ is some linear order and $J=(b_i)_{i \in Y}$ is such that $b_i\models p\restriction AIb_{<i}$, then $I+J$ is $A$-indiscernible. \end{lem}
\begin{proof} We may assume that $|Y|<\omega$ and prove by induction on $|Y|$, so that it is enough to prove it in the case $|Y| = 1$. So suppose that $b\models p\restriction AI$, and we should show that $I+b$ is $A$-indiscernible. Indeed, suppose that $i_0< \dots <i_{n-1} \in X$, $\varphi(x_0,\dots ,x_{n-1},y)\in\mathcal{L}(A)$ and $\models\varphi(a_{i_0},\dots, a_{i_{n-1}}, b)$. Since $b\models p \restriction AI$, $\varphi(a_{i_0}, \dots, a_{i_{n-1}}, x)\in p(x)$. Fix some $a_{j_0},\dots , a_{j_n} \in X$ (which exist since $X$ is infinite). Note that $a_{i_0},\dots , a_{i_{n-1}}\equiv^{Ls}_A a_{j_0},\dots , a_{j_{n-1}}$ because $I$ is infinite. Since $p$ is Lascar-invariant over $A$, $\varphi(a_{j_0},\dots , a_{j_{n-1}}, x)\in p$. Therefore $\models\varphi(a_{j_0},\dots , a_{j_{n-1}}, a_{j_n})$.\end{proof}



\begin{lem}\label{lem:expanding Morely sequence} Suppose that $p(x)\in S(\mathfrak{C})$ is a global $A$-Lascar-invariant type. Suppose that $I = I_1 + I_2$ is an infinite Morley sequence generated by $p$ over $A$. Then, if $J$ is such that $I' = I_1 + J +I_2$ is $A$-indiscernible, then $I'$ is a Morley sequence generated by $p$ over $A$. \end{lem}

\begin{proof}
    Suppose that $a \in I'$ and that $\varphi(a,b)$ holds for $b \in I'_{<a}$ and $\varphi$ a formula over $A$. Let $a'b' \in I$ be such that $b'\in I_{<a'}$ and $a'b' \equiv_{A}^{Ls} ab$ (just let $a'b'$ be any increasing tuple from $I$ of the same length). Then $\varphi(a',b')$ holds, so that $\varphi(x, b') \in p$ and by Lascar invariance $\varphi(x,b) \in p$.
\end{proof}

The following is an easy lemma which will be useful in Corollary \ref{cor:niptotind}.
\begin{lem}\label{lem:lema} Suppose that $M\supseteq A$. Let $p\in S(\mathfrak{C})$ be a global type which is Lascar-invariant over $A$, $J$ an infinite $A$-indiscernible sequence and $a\models p\restriction M$.

Then, we can find $I=(a_i)_{i<\omega}$ such that $a+I$ is a Morley sequence generated by $p$ over $A$ and $J$ is $AI$-indiscernible.

Moreover, if $J = (J_\alpha)_{\alpha<\kappa}$ is a sequence of infinite $A$-mutually indiscernible sequence then we can find some $I$ as above such that $J$ is $AI$-mutually indiscernible.
\end{lem}
\begin{proof} 
    Let $I = (a_i)_{i<\omega}$ be a Morley sequence generated by $p$ over $MaJ$. Then $I \models p^{(\omega)} \restriction MaJ$ and $a+I \models p^{(\omega+1)} \restriction M$ so it is $M$-indiscernible. Since $p^{(\omega)}$ is Lascar-invariant it implies that $J$ is $AI$-indiscernible.
    
    The ``moreover'' part is proved similarly, since in a sequence of mutually indiscernible sequences, every two finite tuples multi-ordered in the same way have the same Lascar-strong type. 
\end{proof}


\begin{defi} We say that a global type $p$ is \emph{generically stable over} $A$ if it is $A$-invariant and for every ordinal $\alpha$, every $\varphi(x)$ with parameters in the monster model and every Morley
sequence $(a_i)_{i<\alpha}$ generated by $p$ over $A$, the set $\Set{i<\alpha}{\models\varphi(a_i)}$ is finite or cofinite.\end{defi}
\begin{fact} \label{fac:generically stable -> totally indiscernible} \cite[Proposition 2.1]{AnandPredrag} If $p$ is generically stable over $A$ then every Morley sequence $(a_i)_{i < \omega}$ generated by $p$ over $A$ is totally indiscernible over $A$.\end{fact}



\section{Preservation results}
\label{sec:results}

\subsection{Preservation under non-co-dividing extensions}

It turns out that if in Question \ref{question} we ask about co-dividing instead of dividing, this is extremely easy.  In particular this proves that the restriction of a stable type, non-forking over a model $M$, to $M$, is itself stable (this was first shown in \cite[Claim 2.19]{HassonOnshuus}).

\begin{defi}
A type $\tp(a/B)$ \emph{co-divides} (\emph{co-forks}) over $A\subseteq B$ if $\tp(B/aA)$ divides (forks) over $A$. 
\end{defi}

We will need the following well-known fact about dividing.
\begin{fact} \label{fac:equivaelnce of dividing} \cite[Corollary 7.1.5]{TZ}
    Suppose that $\tp(c/Cd)$ does not divide over $C$. Then for every infinite $C$-indiscernible sequence $J$ containing $d$, there is a $Cc$-indiscernible sequence $J'$ which is indiscernible over $Cc$ and satisfies $J'\equiv_{Cd}J$.
\end{fact}

\begin{prop}\label{pro:positive1} Let $p\in S(B)$ be an NIP (stable) type which does not co-divide over $A\subseteq B$, then $p\restriction A$ is NIP (stable). \end{prop}
\begin{proof} We show the NIP case; the proof for the stable case is exactly the same argument with OP instead IP. Let $p=\tp(a/B)$ which does not co-divide over $A\subseteq B$. If $\tp(a/A)$ has IP, then by Lemma \ref{lem:pnip} there is an $A$-indiscernible sequence $I=(a_i)_{i<\omega}$ of realizations of $\tp(a/A)$, with $a_0=a$, a tuple $b$ and a formula $\varphi(x,y)$ such that $\models\varphi(a_i,b)^{(i\,\mbox{even})}$ for all $i<\omega$. By Fact \ref{fac:equivaelnce of dividing}, there is a $B$-indiscernible sequence $I'$ such that $I'\equiv_{Aa} I$. Conjugating, we find a tuple $b'$ such that $\models\varphi(a'_i,b')^{(i\,\mbox{even})}$ for all $i<\omega$. Together we have that $I'$ is a $B$-indiscernible sequence of realizations of $\tp(a/B)$ witnessing IP for $\tp(a/B)$, a contradiction.\end{proof}

For dp-rank the situation is a bit more involved. 
\begin{prop} \label{prop:preserving dp-rank under co-forking} Suppose that $p \in S(B)$ does not co-fork over some $A \subseteq B$. Then $\dpr(p) = \dpr(p\restriction A)$. \end{prop}
\begin{proof}
    We may assume that $p \restriction A$ is NIP, as otherwise by Proposition \ref{pro:positive1}, $p$ has IP, so its dp-rank is $\infty$ and there is nothing to prove. 

    Suppose that $\dpr(p\restriction A) \geq \kappa$. By Remark \ref{rem:ict pattern}, this means (by compactness) that there is an ict-pattern of depth $\kappa$ as witnessed by a sequence of formulas $(\varphi_\alpha(x,y_\alpha))_{\alpha<\kappa}$, and an array of tuples $(a_{i,\alpha})_{i<\lambda,\alpha<\kappa}$ with $\lambda = |T|^+$. So for each $\eta:\kappa \to \lambda$ there is some $d_\eta \models p\restriction A$ with $d_\eta \models \varphi_\alpha(x,a_{i,\alpha})$ iff $\eta(\alpha)=i$. By a similar argument to \cite[Lemma 5.6]{Ch}, we may assume that $\tp(a_{i,\alpha}/A)$ is NIP. Namely, let $F = \Set{d_\eta}{\eta:\kappa \to \lambda}$. By honest definitions \cite[Proposition 1.1]{ArtemPierreI}, there are formulas $\theta_{\alpha,i}(x,b_{\alpha,i})$ such that $b_{\alpha,i}$ are tuples of realizations of $p\restriction A$ and such that $\theta_{\alpha,i}(F,b_{\alpha,i}) = \varphi_{\alpha}(F,a_{\alpha,i})$. We may assume that $\theta_{\alpha,i} = \theta_\alpha$ constantly. Then $(\theta_\alpha)_{\alpha<\kappa}$ and $(b_{\alpha,i})_{\alpha<\kappa,i<\omega}$ witness an ict-pattern for $p\restriction A$ (with the same $d_\eta$). Since $b_{\alpha,i}$ are tuples of realizations of $p\restriction A$ which is NIP, it follows from Fact \ref{fac:tupnip} that their type over $A$ is NIP.
    Extracting, we may assume that the sequences $((b_{\alpha,i})_{i<\omega})_{\alpha<\kappa}$ are mutually indiscernible over $A$.
    
    By applying an automorphism, and perhaps changing the $b_{\alpha,i}$'s, we can thus find some $d \models p\restriction A$ such that $d \models \theta_\alpha(x,b_{\alpha,i})$ iff $i = 0$ for all $\alpha < \kappa$ and $\tp(B/dA)$ does not fork over $A$. By taking a non-forking extension (and applying an automorphism), we may assume that $\tp(B/d(b_{\alpha,0}b_{\alpha,1})_{\alpha<\kappa})$ does not fork over $A$. Let $I = ((b_{\alpha,2i}b_{\alpha,2i+1})_{\alpha<\kappa})_{i<\omega}$, so $I$ is an $A$-indiscernible sequence. By Fact \ref{fac:equivaelnce of dividing}, there is a sequence $I' \equiv_{A(b_{\alpha,0}b_{\alpha,1})_{\alpha<\kappa}} I$ such that $I'$ is $B$-indiscernible. Write $I' = ((b'_{\alpha,2i}b'_{\alpha,2i+1})_{\alpha<\kappa})_{i<\omega}$.
    
    For $\alpha<\kappa$, let $J_\alpha = (b'_{\alpha,i})_{i<\omega}$. Then the sequences $(J_{\alpha})_{\alpha<\kappa}$ are mutually indiscernible over $A$, and the type of all elements in $J_\alpha$ over $A$ is NIP, so $(J_{\alpha})_{\alpha<\kappa}$ are mutually indiscernible over $B$ (this is an exercise left to the reader, see \cite[Exercise 4.10]{Sim} and also \cite[Lemma 5.2]{Ch}). 
    
    Clearly, no $J_\alpha$ is indiscernible over $d$, and hence $\dpr(p) \geq \kappa$. 
\end{proof}

In the proof above, we could have used Fact \ref{fac:dp rank witnessing} instead of honest definitions to get that the sequences witnessing the dp-rank have NIP types, but we thought that the argument above is a bit nicer. 

Another remark is that if the type $\tp(B/aA)$ is NIP (\textit{e.g.} if $T$ is NIP), then the proof is much easier: if $I =(I_\alpha)_{\alpha<\kappa}$ are mutually indiscernible over $A$, none of them indiscernible over $Aa$, then we may assume, by taking a non-forking extension, that $\tp(B/AaI)$ does not fork over $A$. But then by Lascar-non-splitting (see Lemma \ref{lem:lascar}), it follows that $(I_\alpha)_{\alpha<\kappa}$ are mutually indiscernible over $B$.

\begin{cor}\label{cor:positive2} 
    Let $p=\tp(a/A)$ be an NIP (stable) type which is definable over some model $M\subseteq A$. Then $p\restriction M$ is NIP (stable) and has the same dp-rank.
\end{cor}
\begin{proof} 
    Since $p$ is definable (over $M$), it is easy to see that it does not co-fork over $M$. Just note that $p$ is a heir over $M$ so that $\tp(A/Ma)$ is a coheir over $M$ and hence it does not fork over $M$. Therefore the result follows from the propositions above.
\end{proof}

By Lemma \ref{lem:lascar} (6) and taking a global non-forking extension, we get:
 \begin{cor} \label{cor:stable types nf over models}
    If $p\in S(A)$ is a stable type which does not fork over a model $M\subseteq A$, then its restrictions to $M$ is stable.
 \end{cor}

Recall that the dp-rank is \emph{finite} if $\dpr(x=x) <n<\omega$ for some $n$. In this case, by the sub-additivity of the dp-rank \cite{KOU}, $\dpr(a/A)$ is finite for every finite tuple $a$. Let $\dpr(a/A)$ be the largest $n<\omega$ such that $\dpr(a/A) \geq n$. We say that the dp-rank is \emph{additive} if for any two finite tuples $a$, $b$ and any set $A$, $\dpr(ab/A) = \dpr(a/Ab) + \dpr(b/A)$. 
When the theory is dp-minimal and eliminates $\exists^\infty$, additivity of the dp-rank is equivalent to acl satisfying exchange by \cite[Theorem 0.3]{PierreDprank}.

The following is a very weak preservation result (of dp-rank) for non-forking extensions. 

\begin{cor} \label{cor:if dp-rank is additive then preserved} Suppose that the dp-rank is additive and finite. Then if $a$, $b$ are finite tuples and $\tp(a/Ab)$ does not fork over $A$, then $\dpr(a/A) = \dpr(a/Ab)$. \end{cor}

\begin{proof}
    By additivity, we know that (*) $\dpr(a/Ab) + \dpr(b/A) = \dpr(a/A)+\dpr(b/Aa)$. As $\tp(b/Aa)$ does not co-fork over $A$, Proposition \ref{prop:preserving dp-rank under co-forking} implies that $\dpr(b/Aa) = \dpr(b/A)$. Eliminating it from both sides of (*), we are done.
\end{proof}

\subsection{Preservation of NIP and dp-rank}
In this section we prove our main positive result as discussed in the introduction, but first we have the following straightforward observation, answering Question \ref{question} positively when the theory is either NIP or NSOP. 

\begin{rk}\label{rem:nsop} Let $T$ be an NIP or NSOP theory and $p\in S(B)$ be an NIP type which does not fork over $A\subseteq B$, then $p\restriction A$ is NIP. \end{rk}
\begin{proof} If $T$ is NIP, all types are NIP.

If $T$ is NSOP, then NIP types are stable types because of Remark \ref{rem:stsop}. Hence the result follows, since we know by \cite{Gen} that stability is preserved under non-forking restriction (this will also be proved in Section \ref{sec:preservation of stability}).  \end{proof} 

The following is the main theorem of this section.

\begin{theo}\label{the:result1} Let $p\in S(\mathfrak{C})$ be an NIP global type which does not fork over $A$ and let $I = (a_i)_{i<\omega}$ be a Morley sequence generated by $p$ over $A$, then $p\restriction AI$ is NIP.\end{theo}
\begin{proof} Assume towards a contradiction that $p\restriction AI$ has IP as witnessed by $\varphi(x,y)$, \textit{i.e.} there is some $AI$-indiscernible sequence $J=(e_i)_{i<\omega}$ and $f$ realizing $p\restriction AI$ such that $\models\varphi(f,e_i)^{(i\,\mbox{even})}$ for all $i<\omega$. 
    
Since $p$ is NIP, by Remark \ref{rem:pglobnip} applied to $\varphi(x,y)$ there is a formula $\sigma(x,z):=\alpha_{\varphi}(x,z)\in\mathcal{L}$ with $\sigma(x,c)\in p$, such that $\varphi(x,y) \land \sigma(x,c)$ is NIP. Applying Remark \ref{rem:pglobnip} to $\sigma(x,z)$ we find a formula $\chi(x,u):=\alpha_{\sigma}(x,u)\in\mathcal{L}$ with $\chi(x,d)\in p$, such that $\sigma(x,z) \land \chi(x,d)$ is NIP.


 We may assume that $\models\chi(a_i,d)$ for all $i<\omega$. Indeed, let $J = (b_{i})_{i<\omega}$ be a Morley sequence generated by $p$ over $AId$. Since $p$ is non-forking over $A$ it is Lascar-invariant over $A$ by Lemma \ref{lem:lascar}, thus by Lemma \ref{lem:indlas2}, $I+J$ is $A$-indiscernible. Hence $I\equiv^{Ls}_A J$. By applying a Lascar-strong automorphism over $A$, there is some $d'$ such that $d'I\equiv^{Ls}_A dJ$. Hence, $\chi(a_i,d')$ holds for all $i<\omega$, $\chi(x,d')\in p$ (by Lascar-invariance) and $\sigma(x,z) \land \chi(x,d')$ is NIP, so replacing $d$ with $d'$ we are done.


\begin{cla} For every $n<\omega$ there is $c_n\equiv^{Ls}_{A} c$ such that for all $i<n$, $\models\sigma(a_i,c_n)^{(i\,\mbox{even})}$.\end{cla}
\begin{clmproof}
The proof is by induction on $n$. The case $n=0$ is trivial, so assume we have $\Set{c_i}{i\leq n}$ and we are going to find $c_{n+1}$.

 In case $n$ is even, let $a\models p\restriction AIc_{\leq n}$. It follows that $\models\sigma(a,c_n)$ because $\sigma(x,c_n)\in p$ as $p$ is Lascar-invariant over $A$ and $c_n\equiv^{Ls}_{A} c$. By Lemma \ref{lem:indlas2}, $I+a$ is $A$-indiscernible, hence $a\equiv^{Ls}_{Aa_{<n}}a_n$. By applying a Lascar-strong automorphism, we find $c_{n+1}$ such that $c_{n+1}\equiv^{Ls}_{Aa_{<n}}c_n$ (so also over $A$) and $\models\sigma(a_n,c_{n+1})$ as we want. Note that we still have that $\models\sigma(a_i,c_{n+1})^{(i\,\mbox{even})}$ for $i<n$ by the induction hypothesis.
 
 Assume $n$ is odd. As $a_n \equiv^{Ls}_{Aa_{<n}} f$ (by Lemma \ref{lem:indlas2}), by applying a Lascar-strong automorphism over $Aa_{<n}$ we can find $J'=(e_i')_{i<\omega}$ such that $J'a_n\equiv^{Ls}_{Aa_{<n}}Jf$. 
 Applying Lemma \ref{lem:lascarseq} (which we can since $J'$ is indiscernible over $Aa_{<n}$), there is $J''\equiv^{Ls}_{Aa_{<n}}J'$ indiscernible over $Aa_{<n}c_n$. 
 By applying a Lascar-strong automorphism, there is $c_{n+1}$ satisfying $c_{n+1}J'\equiv^{Ls}_{Aa_{<n}}c_nJ''$ and therefore $J'$ is indiscernible over $Aa_{<n}c_{n+1}$. It follows that $\models\lnot\sigma(a_n,c_{n+1})$. 
 This is because (1) $\sigma(x,c_{n+1})\land \varphi(x,y)$ is NIP, (2) $J'$ is $c_{n+1}$-indiscernible and (3) $\models\varphi(a_n,e'_i)^{(i\,\mbox{even})}$ for all $i<\omega$ by the choice of $J'$. 
 Since $c_{n+1}\equiv^{Ls}_{Aa_{<n}}c_n$ we are done by the induction hypothesis. \end{clmproof}


 By compactness there is $c_{\omega}$ satisfying $\models\sigma(a_i,c_{\omega})^{(i\,\mbox{even})}$ for all $i<\omega$. However, since $I$ is indiscernible and $\models \chi(a_i,d)$ for all $i<\omega$, we get a contradiction to the fact that $\chi(x,d) \land \sigma(x,z)$ is NIP (see Remark \ref{rem:pglobnip}).\end{proof}

The difference between the next corollary and Theorem \ref{the:result1} is that in Theorem \ref{the:result1} the sequence $I$ was of order type $\omega$.
 
\begin{cor} \label{cor:main positive NIP result}
    Let $p\in S(\mathfrak{C})$ be an NIP global type which does not fork over $A$ and let $I$ be an infinite Morley sequence generated by $p$ over $A$, then $p\restriction AI$ is NIP.
\end{cor}

\begin{proof}
    Let $J=(a_i)_{i<\omega}$ be a Morley sequence generated by $p$ over $AI$. By Theorem \ref{the:result1}, we know that $p \restriction AJ$ is NIP. Let $\varphi(x,y)$ be a formula. By Remark \ref{rem:pglobnip}, there is a formula $\alpha_{\varphi}(x,z)$ and a tuple $c \in AJ$ such that $\alpha_{\varphi}(x,c)\in p$ and $\alpha_{\varphi}(x,c) \land \varphi(x,y)$ is NIP. As $I+J$ is indiscernible over $A$ (by Lemma \ref{lem:indlas2}), ``moving'' $c$ into $AI$ we find some $c' \equiv^{Ls}_A c$ such that $c'\in AI$. Thus  $\alpha_{\varphi}(x,c') \land \varphi(x,y)$ is NIP and $\alpha_{\varphi}(x,c') \in p$ (by Lascar-invariance). 
    Since $\varphi$ was arbitrary, this proves that $p \restriction AI$ is NIP by Remark \ref{rem:pglobnip}.
\end{proof}

\begin{cor}\label{cor:result2}  Let $p\in S(B)$ be an NIP type which does not fork over $A\subseteq B$. If there is a Morley sequence $I\subseteq B$ generated by $p$ over $A$, then $p\restriction AI$ is NIP. \end{cor}
\begin{proof} By Corollary \ref{cor:main positive NIP result}, taking a global non-forking extension of $p$. \end{proof} 

Next we generalize Corollary \ref{cor:main positive NIP result} for dp-rank. 

\begin{theo} \label{the:preserving dp-rank} Let $p\in S(B)$ be a type which does not fork over $A\subseteq B$. If there is a Morley sequence $I\subseteq B$ generated by $p$ over $A$, then $\dpr(p\restriction AI) = \dpr(p)$.
\end{theo}

Before the proof, note that By Remark \ref{rem:dp-rank infinite -> IP}, this theorem implies Corollary \ref{cor:main positive NIP result}. Even though Corollary \ref{cor:main positive NIP result} will not be used in the proof, we thought it is better to include both proofs since the proof of Theorem \ref{the:preserving dp-rank} is basically the same as that of Corollary \ref{cor:main positive NIP result} with some extra complications. 

\begin{proof}[Proof of Theorem \ref{the:preserving dp-rank}]
    By taking a global non-forking extension, we may assume that $p$ is global. First we prove it in the case when $I = (a_i)_{i<\omega}$. 

    Suppose that $\dpr(p \restriction AI) \geq \kappa$ as witnessed by $f \models p \restriction AI$ and $(I_{\alpha})_{\alpha<\kappa}$, while $\dpr(p) < \kappa$. 
    Since $I_\alpha$ is not indiscernible over $AIf$, there is a formula $\varphi_\alpha(x,y_\alpha)$ over $AI$ and two increasing sequences $e_{\alpha,0},e_{\alpha,1}$ witnessing this: $\models \varphi_\alpha(f,e_{\alpha,0})\land \neg \varphi_\alpha(f,e_{\alpha,1})$.
    Let $\Gamma(x,(z_\alpha)_{\alpha<\kappa})$ be the (global) partial type in infinitely many variables saying that $(z_\alpha)_{\alpha < \kappa}$ is a sequence of $\mathfrak{C}$-mutually indiscernible sequences of the same type as $(I_\alpha)_{\alpha <\kappa}$ over $\emptyset$ such that each $z_\alpha$ is not indiscernible over $AIx$ as witnessed by $\varphi_\alpha$ and $z_{\alpha,0},z_{\alpha,1}$ as above (so $z_{\alpha,0},z_{\alpha,1}$ correspond to the same location that $e_{\alpha,0},e_{\alpha,1}$ has). 
    As $\dpr(p) < \kappa$, $\Gamma \cup p(x)$ is inconsistent, which means that there is some formula $\sigma(x,c)$ in $p$, and finitely many sequences and formulas which without loss of generality are $(\varphi_k,I_k)_{k<l}$ such that (*) we cannot find $f' \models \sigma(x,c)$ and $c$-mutually indiscernible sequences $(J_k)_{k<l} \equiv (I_k)_{k<l}$ such that each $J_k$ is not indiscernible over $f'AI$ as witnessed by $\varphi_k$ and the appropriate tuples. Since there are only finitely many formulas, they use only finitely many parameters from $I$, so by incorporating those parameters into $A$ and re-enumerating $I$, we can assume that all of those formulas are over $A$.
    It follows that if $c' \equiv_A c$ then (*) is true with $c$ replaced by $c'$ as well.     

    Note that if $p$ has IP, there is nothing to prove since the same is true for any restriction, so we may assume that $p$ is NIP, in which case, by Remark \ref{rem:pglobnip}, there exists a formula $\chi(x,u):=\alpha_{\sigma}(x,u)$  and some tuple $d$ such that $\chi(x,d) \in p$ and $\chi(x,d) \land \sigma(x,u)$ is NIP, as in the proof of Theorem \ref{the:result1}. As in that proof, we may assume that for every $i<\omega$, $\chi(a_i,d)$ holds. 

    Now the proof proceeds similarly to the proof of Theorem \ref{the:result1}. Namely, we prove the same claim as in there: by induction on $n<\omega$, we find $c_n$ such that $c_n\equiv^{Ls}_{A} c$ and for all $i<n$, $\models\sigma(a_i,c_n)^{(i\,\mbox{even})}$.

    The proof of the claim is similar, the even case is exactly the same, and in the odd case, letting $J = (I_k)_{k<l}$, we can repeat the same proof, this time using the ``moreover'' part of Lemma \ref{lem:lascarseq}. 

    Finally, if $I$ is any infinite sequence (not necessarily of order type $\omega$), then we apply the same proof of Corollary \ref{cor:main positive NIP result} with some slight complications. Start by forgetting the notation from the argument above. 
    
    As $I$ is infinite, we can write $I = I_1 + I_2$ such that either $I_1$ has no end or $I_2$ has no beginning (one of them could be empty). 
    By compactness and applying an automorphism, we can find a sequence $J=(a_i)_{i<\omega}$ such that $I' = I_1 + J + I_2$ is $A$-indiscernible. By Lemma \ref{lem:expanding Morely sequence}, it follows that $I'$ is a Morley sequence generated by $p$ over $A$. 
    By the previous case, it is enough to show that $\dpr(p \restriction AI) = \dpr(p \restriction AJI)$. 
    
    Suppose that $\dpr(p \restriction AJI) < \kappa$ while $\dpr(p \restriction AI) \geq \kappa$ as witnessed by a sequence $(I_\alpha)_{\alpha<\kappa}$ of $AI$-mutually indiscernible sequences, a realization $d \models p \restriction AI$, formulas $\varphi(x,y_\alpha)$ and tuples as above. Applying the same compactness argument as above, we find some $\sigma(x,c) \in p\restriction AIJ$ and finitely many sequences and formulas which without loss of generality are $(\varphi_k,I_k)_{k<l}$ such that (*) from above holds. 
    Let $A_0 = AI_*$ where $I_* \subseteq I$ is some finite subset, containing all parameters appearing in the formulas $(\varphi_k)_{k<l}$.
    Note that if $c' \equiv_{A_0} c$ then (*) holds for $c'$.

    By indiscernibility and choice of $I_1$ and $I_2$, ``moving'' $c$ into $AI$ we can find some $c' \equiv^{Ls}_{A_0} c$ such that $c'\in AI$. So $\sigma(x,c') \in p \restriction AI$ and (*) holds for $c'$. But then the array $(I_\alpha)_{\alpha<\kappa}$ and $d$ contradicts (*) for $c'$. 
\end{proof}


\subsection{Some applications} \label{sec:applications}

As promised in the introduction, we give some applications. 

In the first one we prove the local version of the following fact, see \cite[Proposition 2.11]{HP}: in an NIP theory if $p$ is a global NIP type then $p$ is Lascar-invariant over $A$ if and only if $p$ is Kim-Pillay invariant over $A$. 
Trying to mimic the proof from \cite{HP} we realized that we need more than just $p$ being NIP, and that Theorem \ref{the:result1} is exactly what we need to complete the argument. 

\begin{cor} \label{cor:Lascar=KP} Let $N \models T$ be Lascar complete over $A$ and $p\in S(N)$ an NIP type. Then $p$ does not Lascar-split over $A$ if and only if $p$ does not KP-splits over $A$. \end{cor}
\begin{proof} 
    The direction from right to left is always true since having the same Lascar-strong type implies having the same Kim-Pillay strong type.
    Assume $p$ does not Lascar-split over $A$. Since $N$ is Lascar-complete, we can extend $p$ to a global type which is Lascar-invariant over $A$ (Lemma \ref{lem:lascar}), and this type will still be NIP. Hence we may assume that $p$ is global. We follow the presentation in \cite{Cas14}, namely Proposition 5 there. We summarize the adaptations of the auxiliary results, which appear in \cite{Cas14}, the proofs are the same as there. 
    
    We assume that $\tau$ is a KP-strong automorphism over $A$, and we wish to show that $\tau(p)=p$. For this we want to use the following lemma, which is Lemma 6 from \cite{Cas14} adapted.

    $\bullet$ Let $p$ be a global NIP type, Lascar-invariant over $A$ and let $\tau\in Aut(\mathfrak{C}/A)$, if for each $n<\omega$ and $a\models p^{(n)}\restriction A$ we have $a\equiv_A^{Ls}\tau a$, then $\tau(p)=p$.
 
    Lemma 6 there follows from Lemma 5 (the key step there) which becomes:

    $\bullet$ Let $p_1, p_2$ be global types, Lascar-invariant over $A$ and let $I\models p^{(\omega)}_1\restriction A$ be such that $p_1\restriction AI=p_2\restriction AI$ is NIP, then $p_1=p_2$.
    
    When we use Lemma 5 in the proof of Lemma 6, applying Theorem \ref{the:result1} we precisely get that $p \restriction AI$ is NIP and the proof goes through.
    
    Once we have Lemma 6, we need to show that for any  $a\models p^{(n)}\restriction A$, $\tau(a) \equiv_A^{Ls} a$. We apply Theorem 4 from there which remains the same:

    $\bullet$ If $q\in S(\mathfrak C)$ is Lascar-invariant over $A$, then for every $a,b\models q$ we have $a\equiv_A^{Ls}b$ if and only if $d_A(a,b)\leq 2$. 

    In particular, this is true for $q = p^{(n)}$. Finally, since $a \equiv_A^{KP} \tau(a)$, we are done.
    
    (Note that we use the fact that having the same KP-strong type over $A$ refines any $A$-type-definable equivalence relation on an $A$-definable type, see \cite[Proposition 15.25]{cassanovasBook}, so that since $\mathord{\equiv^{Ls}_A}$ is type-definable on realizations of $q \restriction A$, it equals $\equiv_A^{KP}$ there.)
\end{proof} 


Next we show that if $p$ has the property that every (some) Morley sequence generated by $p$ over $A$ is totally indiscernible, then the answer to Question \ref{question} is positive, and moreover, the dp-rank is preserved.

We first show the following (this is probably well-known, but we give a proof). 
\begin{prop} \label{pro:totally indiscernible over}
    Suppose that $p(x) \in S(A)$ is an NIP type and that $I$ is an infinite $A$-indiscernible sequence of realizations of $p$. Then, if $I$ is totally indiscernible over $\emptyset$ then it is totally indiscernible over $A$. 
\end{prop}
\begin{proof}
    By compactness, we may assume that $I = (a_i)_{i<\omega}$ \textit{i.e.} has order-type $\omega$. 

    Some notation: for any $n$-tuple $c = (c_i)_{i<n}$ and for any permutation $\sigma$ of $n$, let $c^\sigma$ be the tuple $(c_{\sigma(i)})_{i<n}$. 
    
    Suppose that $I$ is not totally indiscernible over $A$. Then for some $n<\omega$, some permutation $\sigma$ of $n$, some tuple $c\in A$, and some formula $\theta(z,c)$ we have that $\models \theta(a_{<n},c) \land \neg \theta(a_{<n}^\sigma,c)$. By indiscernibility over $A$, we have that the same is true when we replace $a_{<n}$ by $(a_i)_{kn \leq i < (k+1)n}$ for any $k<\omega$. Thus we can produce an $\emptyset$-indiscernible sequence $K = (d_k)_{k<\omega}$ such that $\models \theta(d_k,c)^{(k\,\mbox{even})}$ for all $k<\omega$, namely $d_k = (a_i)_{kn \leq i < (k+1)n}$ for $k$ even and $d_k = ((a_i)_{kn \leq i < (k+1)n})^\sigma$ for $k$ odd (it is $\emptyset$-indiscernible because $I$ is totally indiscernible over $\emptyset$).

    Now consider $K' = (d'_k)_{k<\omega}$ where $d'_k = d_{2k}$. Then $K'$ is $A$-indiscernible and $K' \equiv_\emptyset K$. Thus for some $c'$, we have that $c'K' \equiv_\emptyset cK$. In particular, $\models \theta(d_k',c')^{(k\,\mbox{even})}$ for all $k<\omega$. 

    By Fact \ref{fac:tupnip}, $\tp(d_0'/A)$ is NIP. Together we have a contradiction. 
\end{proof}

\begin{cor} \label{cor:totally indiscernible over} 
    Suppose that $p$ is a global NIP type which does not fork over $A$, and that $I$ is an infinite Morley sequence generated by $p$ over $A$. Then if $I$ is totally indiscernible (over $\emptyset$) then it is totally indiscernible over $A$.
\end{cor}
\begin{proof} 
    Suppose that $p$, $A$ and $I$ is as above and that $I$ is totally indiscernible over $\emptyset$. let $J = (a_i)_{i<\omega}$ be a Morley sequence generated by $p$ over $AI$. Then $I+J$ is $A$-indiscernible by Lemma \ref{lem:indlas2}, so that they have the same EM-type over $A$ and in particular $J$ is totally indiscernible over $\emptyset$. By the same reason it is enough to show that $J$ is totally indiscernible over $A$. But by Proposition \ref{pro:totally indiscernible over} and Corollary \ref{cor:main positive NIP result}, $J$ is totally indiscernible over $AI$, so we are done.
\end{proof}

\begin{cor}\label{cor:niptotind} Let $p \in S(\mathfrak{C})$ be an NIP global type which does not fork over $A$ and such that (some) every infinite Morley sequence $I$ generated by $p$ over $A$ is totally indiscernible, then $p\restriction A$ is NIP and moreover $\dpr(p) = \dpr(p\restriction A)$.\end{cor}
\begin{proof} 
    Note that all infinite Morley sequences generated by $p$ over $A$ have the same EM-type over $A$ (by Lemma \ref{lem:indlas2}: given two such sequences, generate new one over their union), so if one is totally indiscernible so is any other.

    Assume towards a contradiction that $p\restriction A$ has IP, so there are $\varphi(x,y)\in\mathcal{L}$, $J=(b_i)_{i<\omega}$ indiscernible over $A$ and $a\models p\restriction A$ such that $\models\varphi(a,b_i)^{(i\,\mbox{even})}$ for all $i<\omega$.

    We may assume that $a\models p\restriction M$ for some model $M\supseteq A$. Otherwise replace $a$ by some $a' \models p \restriction M$, and replace $J$ be some $J'$ such that $J'a' \equiv_A Ja$.

    By Lemma \ref{lem:lema} we can find $I=(a_i)_{i<\omega}$ such that $a+I$ is a Morley sequence generated by $p$ over $A$ and $J$ is $AI$-indiscernible. By hypothesis $a+I$ is totally indiscernible so by Proposition \ref{pro:totally indiscernible over} it is totally indiscernible over $A$. But this yields $a\models p\restriction AI$ (let $\psi(x,a_{<i},c) \in p \restriction AI$ with $c \in A$, then for any $j>i$, $\models \psi(a_j,a_{<i},c)$, and by total indiscernibility $\models \psi(a,a_{<i},c)$).

    So we have $a\models p\restriction AI$, an $AI$-indiscernible sequence $J=(b_i)_{i<\omega}$ and $\varphi(x,y)$ such that for all $i<\omega$, $\models\varphi(a,b_i)^{(i\,\mbox{even})}$, which contradicts Theorem \ref{the:result1}.

    The ``moreover'' part (about the dp-rank) is proved exactly the same way, contradicting Theorem \ref{the:preserving dp-rank} using the ``moreover'' part of Lemma \ref{lem:lema}.
\end{proof}

By Fact \ref{fac:generically stable -> totally indiscernible} we get:
\label{cor:generically stable}\begin{cor}  Let $p$ be NIP and generically stable over $A$ then $p\restriction A$ is NIP and of the same dp-rank. \end{cor}

\subsection{Preservation of stability}
\label{sec:preservation of stability}
We recover the main result of \cite{Gen} about the preservation of stability. 

Given $p\in S(\mathfrak{C})$ as in Theorem \ref{the:result1} but stable, by repeating the same proof using OP and IP we can show that $p\restriction AI$ is stable and moreover that $I$ is totally indiscernible. Assume $p\restriction AI$ has OP as witnessed by the formula $\varphi(x,y)$. As $p$ is stable, we use Remark \ref{rem:pglobstable} to get $\sigma(x,c)\in p$ such that $\sigma(x,c) \land \varphi(x,y)$ is stable. Then we use the fact that $p$ is NIP and Remark \ref{rem:pglobnip} to get $\chi(x,d)\in p$ such that $\chi(x,d) \land \sigma(x,z)$ is NIP. Now we proceed as in the proof of Theorem \ref{the:result1} (same claim) towards the contradiction with the choice of $\chi$, this time using that $p\restriction AI$ has OP as witnessed by $\varphi$. This gives us that $p \restriction AI$ is stable. Then, if we take a Morley sequence $J$ generated by $p$ over $AI$, we get that $J$ is totally indiscernible by stability. As in the proof of Corollary \ref{cor:niptotind}, we get that $I$ (and any other Morley sequence generated by $p$ over $A$) is also totally indiscernible. Replacing IP by OP in that proof and taking a global non-forking extension, we finally get:
\begin{cor}\label{cor:stable case, first proof} Let $p \in S(B)$ be a stable type which does not fork over $A \subseteq B$ then $p\restriction A$ is stable.\end{cor}

We give a more direct proof. 

\begin{theo}\label{the:resultstab} Let $p$ be a global stable type which is $A$-invariant, then $p \restriction A$ is stable.\end{theo}
\begin{proof} Assume towards a contradiction that $p \restriction A$ is unstable. By Fact \ref{fac:pstable} there is a formula $\varphi(x,y)$, an $A$-indiscernible sequence $J = (a_i)_{i\in\mathbb{Z}}$ and $b\models p \restriction A$ such that $\models\varphi(b,a_i)^{(i\geq 0)}$. Since $p$ is stable, by Remark \ref{rem:pglobstable} applied to $\varphi(x,y)$, there is a formula $\sigma(x,z):=\alpha_{\varphi}(x,z)\in\mathcal{L}$ such that $\sigma(x,c)\in p$ for some tuple $c$ and $\sigma(x,c) \land \varphi(x,y)$ is stable. Again apply Remark \ref{rem:pglobstable} to $\sigma(x,z)$ to get $\chi(x,u):=\alpha_{\sigma}\in\mathcal{L}$ so that for some $d$, $\chi(x,d)\in p$ and $\chi(x,d) \land \sigma(x,z)$ is stable. 

Without loss of generality, we may assume that $b \models p \restriction Ad$. Indeed, let $b' \models p \restriction Ad$. As $b' \equiv _A b$, there is some automorphism $\tau$ sending $b'$ to $b$ fixing $A$. Let $d' = \tau(d)$, so that $b \models \tau(p \restriction Ad) = p \restriction Ad'$,  $\chi(x,d') \in p$ by invariance and still $\chi(x,d') \land \sigma(x,z)$ is stable. 

\begin{cla} For every $k<\omega$ there are $I_k=(a_{-k},\dots,a_0,\dots,a_k)$, and $c_k\equiv_A c$ such that: \begin{enumerate}
\item $I_k$ is a Morley sequence generated by $p$ over $Ad$. 
\item For all $-k \leq n \leq k$, $\models\sigma(a_n,c_k)$ iff $n \geq 0$.
\end{enumerate}  \end{cla}
\begin{clmproof}
The proof is by induction on $k$. For $k=0$ let $c_0 = c$ and $a_0 \models p \restriction Acd$. Assume we have $I_k$ and $c_k$ satisfying the properties above. We will find $c_{k+1}$, $a_{-(k+1)}$ and $a_{k+1}$ so that $c_{k+1}$ and $I_{k+1} = a_{-(k+1)} + I_k + a_{k+1}$ will work.

 Start with $a_{-(k+1)}$. Let $I'_k$ realize $p^{(2k+1)} \restriction AJbd$. Note that by $A$-invariance of $p$, $J$ is indiscernible over $AI'_k$. Since $I_k\equiv_{Ad}I'_k $, there are $b'$ and $J' = (a'_i)_{i\in\mathbb{Z}}$ such that $b'J'I_k\equiv_{Ad}bJI'_k $, so that $I_k\models p^{(2k+1)} \restriction AJ'b'd$. Therefore we have that $b'+I_k$ realizes $p^{(2(k+1))} \restriction Ad$.

By Ramsey and compactness, let $J''$ be indiscernible over $AI_kc_k$ realizing $EM(J'/AI_kc_k)$. Since $J'$ is indiscernible over $AI_k$, in particular $J'\equiv_{AI_k}J''$. Let $c_{k+1}$ be such that $J'c_{k+1}\equiv_{AI_k}J''c_k$ so that $J'$ is indiscernible over $AI_kc_{k+1}$. Note that $\models\neg\sigma(b',c_{k+1})$. This is because (1) $\sigma(x,c_{k+1}) \land \varphi(x,y)$ is stable, (2) $J'$ is $c_{k+1}$-indiscernible and (3)  $\models\varphi(b',a_i')^{(i\geq 0)}$. Thus, letting $a_{-(k+1)}=b'$, we conclude by induction, since $c_{k+1} \equiv_{AI_k} c_k$.



Finally, let $a_{k+1}\models p \restriction AI_ka_{-(k+1)}c_{k+1}d$. It is clear that $I_{k+1} = a_{-(k+1)} + I_k + a_{k+1}$ realizes $p^{(2(k+1)+1)} \restriction Ad$. We also have $\models\sigma(a_{k+1},c_{k+1})$ since $\sigma(x,c)\in p$ and $p$ is $A$-invariant. \end{clmproof}

 Now, by compactness there is an $Ad$-indiscernible sequence $I=(a_i)_{i\in\mathbb{Z}}$ of realizations of $p \restriction Ad$ and $c_{\omega}\equiv_A c$  such that $\models\sigma(a_n,c_\omega)^{(n\geq 0)}$ for all $n \in \mathbb{Z}$ which is a contradiction with the choice of $\chi(x,d)$ (since $\chi(a_n,d)$ holds for all $n \in \mathbb{Z}$). \end{proof}

\begin{cor}\label{cor:resultstab2} Let $p\in S(B)$ be a stable type which does not fork over $A\subseteq B$, then $p\restriction A$ is stable.\end{cor}
\begin{proof} Let $q$ be a global non-forking extension of $p$. By Lemma \ref{lem:lascar} (6), $q$ is $acl^{eq}(A)$-invariant. By Theorem \ref{the:resultstab}, $q\restriction acl^{eq}(A)$ is stable. Finally, by Fact \ref{fac:stableacl} it follows that $p\restriction A$ is stable. \end{proof}

\section{Examples} 
\label{sec:counterexample}

In this section we will introduce a general way of constructing a tree whose family of open cones of any center is a model of some theory. This will be used to produce two examples, one showing that the answer to Question \ref{question} may be negative (using the random graph as the base theory), and another showing that the dp-rank of a non-forking extension of a type may decrease, even in an NIP theory (using the theory of two independent equivalence relations as the base theory). 

We start with some preliminaries on Fra\"iss\'e theory and trees. 

\subsection{Fra\"iss\'e theory}

We assume the reader is familiar with Fra\"iss\'e constructions, QE (quantifier elimination), model companions, and model completions. However, we recall the following results as facts. See for instance sections 7.1 and 7.4 of \cite{Ho} and section 3.5 of \cite{C-K}.

Recall that given a class of structures $\mathcal{K}$, we say that $\mathcal{K}$ is the \emph{age} of a structure $M$ if $\mathcal{K}$ is the class of structures which are isomorphic to a finitely generated substructure of $M$. We say that $\mathcal{K}$ has HP, the \emph{Hereditary Property}, if for every $A\in\mathcal{K}$ and every substructure $B\subseteq A$, $B\in\mathcal{K}$. $\mathcal{K}$ has JEP, the \emph{Joint Embedding Property}, if for every pair $A, B\in\mathcal{K}$ there is $C\in\mathcal{K}$ such that both $A$ and $B$ embed in $C$. Finally, $\mathcal{K}$ has AP, the \emph{ Amalgamation Property} if for every $A, B, C\in\mathcal{K}$ with $A$ embedding into $B$ and $C$, there is an \emph{amalgam} $D\in\mathcal{K}$ such that $B$ and $C$ embed in $D$ over $A$. More precisely, given embeddings $f_1:A\to B$ and $f_2:A\to C$  there is a structure $D$ in the class $\mathcal{K}$ and embeddings $g_1:B\to D$ and $g_2: C\to D$ such that $g_1\circ f_1=g_2\circ f_2$. If $\mathcal{K}$ has HP, JEP and AP then $\mathcal{K}$ is a \emph{Fra\"iss\'e class}.

Recall also that an $\mathcal{L}$-structure $M$ \emph{ultrahomogeneous} if every isomorphism between two finitely generated substructures extends to an automorphism. 
\begin{rk} \label{rem:weakly homogeneous} By \cite[Lemma 7.1.4]{Ho}, this happens iff for every $A\subseteq M$ and every embedding $f:A \to B$, there is an embedding $g:B\to M$ such that $g \circ f$ is the identity. \end{rk}

\begin{fact} \phantomsection \label{fac:building} \begin{enumerate}

\item{(Fra\"iss\'e's Theorem)} Let $\mathcal{L}$ be a countable language and let $\mathcal{K}$ be an essentially countable\footnote{Having at most countably many structures up to isomorphism.} class of finitely generated $\mathcal{L}$-structures satisfying HP, JEP and AP. Then there is a countable, ultrahomogeneous $\mathcal{L}$-structure $M^*$ unique up to isomorphism such that $\mathcal{K}$ is the age of $M^*$.
 \cite[Thm. 7.1.2]{Ho}.
 $M^*$ is called \emph{the Fra\"iss\'e limit} of $\mathcal{K}$. 
\item Let $\mathcal{L}$ be a finite language and $\mathcal{K}$ an essentially countable uniformly locally finite set of finitely generated $\mathcal{L}$-structures satisfying HP, JEP and AP. Let $M^*$ the Fra\"iss\'e limit and $T=Th(M^*)$, then $T$ is $\omega$-categorical and has QE. \cite[Thm. 7.4.1]{Ho}
\item When $T$ is a countable universal $\mathcal{L}$-theory where $\mathcal{L}$ is finite,  $\mathcal{K}$ is the class of finitely generated models of $T$ and  $\mathcal{K}$ is uniformly locally finite, then the theory $Th(M^*)$ is the model completion of $T$ \cite[Fact 2.1, Remark 2.2]{KaplanSimon}.

 \end{enumerate}\end{fact}

The following remark can be very helpful in simplifying amalgamation arguments. 

\begin{rk}\label{rem:lemaap} Let $\mathcal{K}$ be a class of finitely generated structures. To show that $\mathcal{K}$ has AP, it is enough to check that given $A, B, C\in\mathcal{K}$ such that $A = B \cap C$ and $B=\langle Aa\rangle$ ($|a|=1$) there is an amalgam $D\in\mathcal{K}$. 

Similarly, it is enough to check this in case when both $B,C$ are generated over $A$ by one element. 
\end{rk}
\begin{proof} 
    We are given $A,B,C$ and $f_1,f_2$ as in the definition and we must find $D$ and $g_1,g_2$. Standard techniques allows us to assume that $f_1 = f_2$ is the identify map and that $A = B \cap C$. 

We will prove by induction on $n$ that if $A,B,C$ are as above and  $B=\langle Aa_0,\dots a_{n-1}\rangle$ then there is $D \in \mathcal{K}$ containing $B,C$. 

For $n=0$ there is nothing to prove. 
Suppose that the result is true for $n$, and suppose that $A,B,C$ are as above with  $B=\langle Aa_0,\dots a_{n}\rangle$. Let $B' = \langle Aa_0,\dots a_{n-1}\rangle$, and let $D' \in \mathcal{K}$ be an amalgam of $B',C$, as witnessed by embeddings $f_B' : B' \to D'$ and $f_C: C \to D'$ (such that $f_B' \restriction A = f_C \restriction A$). We may assume that $f_B' = id$ and by perhaps changing $D'$, we may assume that $D' \cap BC = B' \cap BC = B'$.
Let $A^* = B'$, $B^*=B$ and $C^* = D'$. Then we have that $A^* = B^* \cap C^*$ and $B^* = \langle A^*a_n\rangle$. By assumption, there is some amalgam $D \in \mathcal{K}$ and embeddings $g_{B^*}: B^* \to D$, $g_{C^*}: C^* \to D$ such that $g_{B^*} \restriction A^* = g_{C^*} \restriction A^*$. Finally, $g_{B^*}: B \to D$ and $g_{C^*}\circ f_C : C \to D$ are embeddings showing that $D$ is the amalgam of $B,C$ over $A$. 

The last assertion is proved similarly: since we only need to check the case where $C$ is generated by one element over $A$, we induct on the number of generators of $B$ (noting that if $C$ is generated by one element over $A$, letting $D''$ be the substructure of $D'$ generated by $f_C[C]B'$, $D''$ is generated by one element over $A^*$). 
\end{proof}

\begin{defi} \label{def:SAP} A Fra\"iss\'e class $\mathcal{K}$ has the Strong Amalgamation Property (SAP) if in the definition of AP we ask that $Im(g_1) \cap Im(g_2) = Im(g_2\circ f_2)$ (which is $= Im(g_1 \circ f_1)$). We also ask that in JEP, the images of $A$, $B$ are disjoint\footnote{The requirement about JEP was not stated in \cite{Ho}, but it seems to be just an omission, in light of \cite[Theorem 7.1.8]{Ho}.}. 
\end{defi}

SAP allows us to assume that in the context of Remark \ref{rem:lemaap}, the amalgam $D$ contains $B \cup C$. 

\subsection{Trees}\label{sec:trees}

Let $\mathcal{L}_{Tr}=\{\leq, \wedge\}$ and consider the theory $Tr$ of meet-trees, \textit{i.e.} $\leq$ is a partial order such that for every $x$ the set $\{y\leq x\}$ is linearly ordered and for any $x, y$, $\{z\leq x,y\}$ has a greatest element, namely, $x\wedge y$. Since in this paper we have no other type of trees, we will use the term \emph{trees} for meet-trees. 

Given a point $c$ in a tree, we say that $C(c)=\{b\geq c\}$ is the \emph{closed cone of center} $c$. Also consider the following relation defined in the closed cone $C(c)\setminus\{c\}$, $a\mathrel{E}_c b$\,\, if and only if $a\wedge b>c$. Note that $\mathrel{E}_c$ is an equivalence relation. An $E_c$-class is called an \emph{open cone of center} $c$.  
\begin{fact}\label{fac:dt} $Tr$ has a model completion, namely, the theory of dense trees, DT which is determined by saying that: $\leq$ defines a tree where $\wedge$ is the meet operator; for every $x$ the set $\{y\leq x\}$ is a  dense linear order with a maximal element but no minimal one; and for every $x$ there are infinitely many open cones of center $x$. DT is complete, $\omega$-categorical and has QE. It is the theory of the Fra\"iss\'e limit of the class of finite trees. Moreover it is NIP, in fact dp-minimal. (See Section 2.3.1 in \cite{Sim} and Proposition 4.7 in \cite{Sim11}) \end{fact}
\begin{rk} \phantomsection \label{rem:treef} \begin{enumerate}
\item If $A$ is a finite tree and $b$ is a new point, 
 $B=\langle Ab \rangle$ has at most $|A|+2$ points, which are those in $A$, $b$ and possibly a new point $c=\max\Set{b\wedge a}{a\in A}$. 
\item The cardinality of a tree generated by $n$ points is bounded by $2n$.\end{enumerate}\end{rk}  
\begin{proof} (2) follows by induction on (1). 

For (1) note the following two observations: the first is that $\land$ is associative and commutative, and the second is that if $a \land b > a \land c$ then $a \land c = b\land c$.  We leave the details as an exercise to the reader.  \end{proof}

\subsection{A general construction} \label{sec:a general construction}

Suppose that $\Ll$ is a finite relational language, and let $T_{\forall}$ be a universal $\Ll$-theory such that the class of finite models of $T_{\forall}$ is a Fr\"iss\'e class of with SAP. We will construct a new theory whose models are  trees with a generic model of $T_{\forall}$ on the open cones. 

For every $n$-place relation symbol $R$ in $\Ll$, let $R^*$ be an $n+1$-relation symbol, and let $\Ll^*$ be the language $\Ll_{Tr} \cup \Set{R^*}{R \in \Ll}$. Let $T_{\forall}^*$ be the following universal theory. It consists of the axioms $Tr$ for trees, and, for every $n$-place relation symbol $R \in \Ll$, the axioms 
\begin{itemize}
    \item[(\texttt{WD})] $\forall x,y_0,\dots,y_{n-1},y'_0,\dots, y'_{n-1} (R^*(x,y_0, \dots,y_{n-1}) \mathop{\&} \bigwedge_{i<n} x<y_i\wedge y_i') \longrightarrow R^*(x,y'_0, \dots,y'_{n-1})$.
    \item[(\texttt{OC})] $\forall x,y_0,\dots,y_{n-1} R^*(x,y_0, \dots,y_{n-1}) \longrightarrow \bigwedge_{i<n} x < y_i$.
\end{itemize}
(\texttt{WD} stands for Well-Defined and \texttt{OC} for Open Cone.)

We want that every model of $T_{\forall}^*$ will carry a model of $T_{\forall}$ on the open cones of $x$, this is ensured by first adding an axiom describing equality:
\begin{itemize}
    \item[(\texttt{EQ})] $\forall x,y,y' (\mathord{=}^*(x,y,y') \longleftrightarrow x < y \wedge y')$.
\end{itemize} 

Finally, for every quantifier-free formula $\psi(y)$ in $\Ll$ ($y$ a tuple of variables), let $\psi^*(x,y)$ be like $\psi$ but with every instance of a relation symbol $R(y)$ replaced by $R^*(x,y)$ (of course, only the appropriate parts of $y$ appear, in the order coming from $\psi$). As $T_{\forall}$ is universal, every axiom in $T_{\forall}$ has the form $\forall y \psi(y)$. For every such axiom, add to $T_{\forall}^*$ the axiom:
\begin{itemize}
    \item[(\texttt{T})$_\psi$] $\forall x,y (x <y \longrightarrow \psi^*(x,y))$.
\end{itemize} 
(Here $x<y$ means that  $x$ is smaller then all variable in $y$.)

\begin{lem} \label{lem:interpretation} Suppose that $M$ is a model of all axioms of $T_{\forall}^*$ except the axiom scheme (\texttt{T}) and $c\in M$. Then:
    \begin{enumerate}
    \item On the set of open cones of center $c$, if nonempty, we can define an $\Ll$-structure $M_c = \Set{a/E_c}{a\in M,a>c}$ by setting $R^{M_c}(a_0/E_c,\dots,a_{n-1}/E_c)$ iff $R^C(c,a_0,\dots,a_{n-1})$ for any $n$-place symbol $R \in \Ll$ and any $c< a_0,\dots,a_{n-1} \in M$. This gives an interpretation of $M_c$ in $M$ (with parameters).
    \item For any quantifier-free $\Ll$-formula $\psi(y)$ and any tuple $f \in M$ such that $c<f$ (\textit{i.e.} $c$ is smaller than any element from $f$), $M_c \models \psi(f/E_c)$ iff $M\models \psi^*(c,f)$.
    \item For all $c \in M$, $M_c \models T_{\forall}$ iff $M \models \forall y (c<y \to \psi^*(c,y))$ for any axiom $\forall y \psi \in T_{\forall}$.
    \item $M \models T_{\forall}^*$ iff $M_c \models T_{\forall}$ for all $c \in M$.
    \end{enumerate}
\end{lem}
\begin{proof}
    (1)  should be clear. (4) follows from (3) which follow from (2), and (2) is proved by induction. 
\end{proof}

\begin{rk} \label{rem:natural embedding} Note that if $M \subseteq N \models T_{\forall}^*$, then the map $a/E_c \mapsto a/E_c$ from $M_c$ to $N_c$ is an embedding of $\Ll$-structures.
\end{rk}

Let $\Kk$ be the class of finite models of $T_{\forall}^*$.

\begin{prop}\label{pro:amalgamation of T*} $\Kk$ is a uniformly locally finite Fra\"iss\'e class.
\end{prop}

\begin{proof}
    Remark \ref{rem:treef} gives uniformly locally finiteness as $T_{\forall}^*$ has no new function symbols not already in $Tr$. HP is trivial since $T_{\forall}^*$ is a universal.

    In the following arguments, the reader is encouraged to draw pictures. 
    
    For JEP, given disjoint $A, B\in \Kk$, construct $C$ by adding a new point $c$ which is below the roots of $A$ and $B$ with the natural induced tree structure. Let $E \models T_{\forall}$ of size 2, and write $E = \{a^*,b^*\}$. Suppose that $R$ is an $n$-place relation symbol from $\Ll$. For any $a \in A$, ${R^*}^C(a,-) = {R^*}^A(a,-)$ and similarly for $b\in B$, ${R^*}^C(b,-) = {R^*}^B(b,-)$. For any $n$-tuple $f \in AB$, let $f^*$ be the $n$-tuple from $E$ constructed from $f$ by replacing every element from $A$ by $a^*$ and every element from $B$ by $b^*$. Define ${R^*}^C(c,f)$ iff $R^E(f^*)$. Note that the same property holds when $R$ is $\mathord{=}$ (where we do not have a choice in the definition by (\texttt{EQ})). By Lemma \ref{lem:interpretation}, since $C_c =E \models T_\forall$, it follows that $C \models T_{\forall}^*$.

    Next we prove AP. By Remark \ref{rem:lemaap} it is enough to consider the case $A,B,C\in \Kk$ with $A = B \cap C$ such that $B=\langle Ab\rangle$. We may assume that $b \notin A$. We will define the amalgam $D$. 

    \underline{Case I}: for some $a\in A$, $b < a$. 

    Let the universe of the $D$ be $Cb$. Let $a_0 = \max\Set{a' \in A}{a'<b}$ (might be $-\infty$) and $a_1 = \min\Set{a' \in A}{a'>b}$. Define the order on $D$ by putting $b$ in the bottom of the interval $(a_0,a_1)$ (so for any $c \in C$ in that interval, $b<c$). (The choice of where to put $b$ in that interval is not important.) Now there is a unique way to define the meet operation, and this defines a meet-tree structure on $D$. 

    Suppose $R \in \Ll$ is an $n$-place relation symbol. The axioms imply that there a unique way to define ${R^*}^D$. Start with defining ${R^*}^D(b,-)$, so let $f \in C$ be an $n$-tuple such that $b < f$ (\textit{i.e.} $b$ is smaller than any element in $f$). It cannot be that $c \wedge a = b$ for any $c \in f$ (as otherwise $b \in C$), so it must be that $c \wedge a > b$, in other words $c \mathrel{E}_b a$. Define ${R^*}^D(b,f)$ iff ${R^*}^B(b,a^n)$ (which we must by (\texttt{WD})). (Here, $a^n$ is the $n$-tuple $(a,\dots,a)$).)

    Now we define ${R^*}(c,-)$ for any $c \in C$. Let $f \in D$ be any $n$-tuple such that $c<f$. Let $f'$ be $f$ replacing each occurrence of $b$ by $a$ (if none exists then $f'=f$), and define ${R^*}^D(c,f)$ iff ${R^*}^C(c,f')$. Again, we have no choice in the matter since if $c<b$, then $b \wedge a = b > c$, so $b \mathrel{E}_c a$.

    Note that in the case that $R$ is $\mathord{=}$ these definitions coincide with the one given by (\texttt{EQ}).

    It is easy to see that under this definition, $B,C \subseteq D$. Now prove by induction that for any quantifier-free formula $\psi$ in $\Ll$, and any $b <f \in C$, $D \models \psi^*(b,f)$ iff $B \models \psi^*(b,a^n)$ and that for any $c \in C$ and $c<f$, $D\models \psi^*(c,f)$ iff $C \models \psi^*(c,f')$, so that $D \models T_{\forall}^*$. 

    \underline{Case II}: not Case I.

    By Remark \ref{rem:treef}, we know that $B = A\cup\{b,b_1\}$ where $b_1=\max\Set{b\wedge a}{a\in A}$. It follows that $b_1 < b$. Let $B' = Ab_1$ (with the induced structure from $B$). Note that $B' \models T_{\forall}^*$, $B' \cap C = A$ and that $b_1 \leq a$ for some $a \in A$. So either $b_1 \in A$ or we are in Case I. Either way, we can put a structure on $C' = Cb_1$ so that it is a model of $T_{\forall}^*$ and $B',C \subseteq C'$. Letting $\bar{A} = B'$, $\bar{B} = B$ and $\bar{C} = C'$, we have that $\bar{A} = \bar{B} \cap \bar{C}$, $\bar{B} = \bar{A} b$ and $b_1'=\max\Set{b\wedge a}{a\in \bar{A}}= b_1$. It is enough to amalgamate $\bar{B}$ and $\bar{C}$ over $\bar{A}$, so this allows us to assume that $b_1 \in A$. 

    Let $D = Cb$, and put a tree structure on $D$ so that $b>b_1$, there is no element from $C$ in between, and $b$ starts a new branch from $b_1$, \textit{i.e.} for every $c \in C$, if $c \not < b_1$, $c \wedge b = c \wedge b_1$. 
    
    Let $R$ be an $n$-place relation symbol from $\Ll$, and we want to define ${R^*}^D$. The only freedom we have is in defining ${R^*}^D(b_1,-)$. Indeed, for any $c \in C$, $c \neq b_1$, and any $n$-tuple $c<f$ from $D$, letting $f'$ be like $f$ but replacing every occurrence of $b$ by $b_1$, by (\texttt{WD}) we must define ${R^*}^D(c,f)$ iff ${R^*}^C(c,f')$ (because $b_1 = b \wedge b_1 > c$ if $b > c$). This definition works also in case when $R$ is $\mathord{=}$ where we do not have a choice by (\texttt{EQ}). Note that for any quantifier-free formula $\psi(y) \in \Ll$, and any tuple $f > c$, $D \models \psi^*(c,f)$ iff $C \models \psi^*(c,f')$.

    Also, ${R^*}^D(b,-)$ must be empty.  

    Next we define ${R^*}^D(b_1,-)$. Let $A^*$ be the $\Ll$-structure $A_{b_1}$ (see Lemma \ref{lem:interpretation}). Let $B^* = B_{b_1}$ (adding one more open cone to $A_{b_1}$, by choice of $b_1$) and $C^* = C_{b_1}$. By Remark \ref{rem:natural embedding}, $A^*$ naturally embeds into $B^*,C^*$ which embed in $D^* = D_{b_1}$ (so far, just as sets). Replace $A^*$, $B^*$ and $C^*$ with their images (with the induced structures), so that we have $A^* = B^* \cap C^*$, $A^* \subseteq B^*,C^*$ as $\Ll$-structures and $D^* = A^* \cup B^*$. As we assumed SAP for the class of finite $T_\forall$-models, and as $A^*,B^*, C^*$ are models of $T_\forall$, we can put an $\Ll$-structure on $D^*$ in such a way that $B^*,C^* \subseteq D^* \models T_\forall$ (note that $A^*$ might be empty, in which case we use JEP as in Definition \ref{def:SAP}). Let $f \in D$ be an $n$-tuple such that $b_1 <f$, and let $f^* \in D^*$ be like $f$ but with every element $c \in C$ replaced by $c/E_{b_1} \in D^*$ and every occurrence of $b$ replaced by $b/E_{b_1} \in D^*$. Define ${R^*}^D(b_1,f)$ iff $R^{D^*}(f^*)$. 


    Note that by (\texttt{EQ}), we do not have a choice in defining $\mathord{=}^*$, but as $D^* = D_{b_1}$ (which we got by the SAP assumption), it follows that using the definition above or the one given by (\texttt{EQ}) amounts to the same thing. 

    Note also that by Lemma \ref{lem:interpretation}, any formula $\psi(y) \in \Ll$, and any tuple $f > b_1$, $D \models \psi^*(b_1,f)$ for any axiom $\forall y \psi(y) \in T_{\forall}$ (since $D^* \models T_\forall$).
    
    It is now easy to see that $B,C \subseteq D$ and that $D\models T_{\forall}^*$ as required. 
\end{proof}

By the last lemma, $\Kk$ satisfies the hypothesis of Fra\"iss\'e's theorem (Fact \ref{fac:building}). 

\begin{defi} Let $T^*$ be the theory of the limit.  \end{defi}

Note that by Fact \ref{fac:building}, we have:
\begin{cor} $T^*$ has QE, it is $\omega$-categorical, and is precisely the model completion of $T_{\forall}^*$.\end{cor}

\begin{lem} \label{lem:types in T*} Suppose that $A$ is some finite substructure contained in a model $M \models T^*_\forall$, and $b\in M$. Let $b_1 = \max\Set{b \wedge a}{a \in A}$. In particular, $b_1 \leq a$ for some $a\in A$. Then the quantifier-free type $\tp_{\qf}(b/A)$ is determined by knowing $a$ and the following: 
    \begin{itemize}
    \item Whether or not $b>b_1$.
    \item Whether or not $b_1 \in A$, and if yes, then which element in $A$.
    \item The order type of $b_1$ in $A_{\leq a}$ (\textit{i.e.} knowing the smallest interval containing it).
    \item For each $n$-place relation symbol $R \in \Ll$, knowing whether or not $R^*(b_1,a^n)$ (where $a^n$ is the $n$-tuple $(a,\dots,a)$).
    \item In case $b>b_1$, the quantifier-free type of $b/{E_{b_1}}$ over $\Set{c/E_{b_1}}{c\in A, c>b_1}$ in the structure $M_{b_1}$ (which is a model of $T_\forall$ by Lemma \ref{lem:interpretation}).
    \end{itemize}
\end{lem}
\begin{proof}
    It is an exercise that the first three bullets determine the quantifier-free type of $b$ over $A$ in the tree language $\{\leq,\wedge\}$. Also, by Remark \ref{rem:treef}, the tree generated by $Ab$ is $Ab_1b$, so for knowing the isomorphism type of $b$ over $A$ (\textit{i.e.} the quantifier-free type), it is enough to determine what is $R^*(c,-)$ for any $c \in Ab_1b$. 
    
    In case that $b_1 = b$, the extra structure is already determined as in Case I of the proof of amalgamation above by the fourth bullet, but we elaborate. If $b \in A$ there is nothing to do so assume not. For every $n$-place relation symbol $R \in \Ll$ and any $c \in A$, and any $n$-tuple $f >c$, let $f'$ be like $f$ but replacing every occurrence of $b$ by $a$, so that by (\texttt{WD}), $R^*(c,f)$ holds iff $R^*(c,f')$ holds, but this is known. Now, for $f>b$, note that for every $c \in f$, $c \wedge a >b$ (otherwise $b\in A$), so $R^*(b,f)$ holds iff $R^*(b,a^n)$ which is again known by the fourth bullet.

    In case where $b_1 < b$, by the previous case we already determined $\tp_{\qf}(b_1/A)$. Let $R$ be as above, and $c \in A \setminus \{b_1\}$. Then for any $n$-tuple $f>c$, $R^*(c,f)$ holds iff $R^*(c,f')$ holds, where $f'$ is like $f$ but replacing every occurrence of $b$ with $b_1$ (note that if $c<b$ then $c<b_1$). Also, $R^*(b,f)$ never holds for $f \in Ab_1b$ by (\texttt{OC}), so we are left to determine $R^*(b_1,-)$. Suppose that $b_1<f$. Then $R^*(b_1,f)$ holds iff $M_{b_1} \models R(f/E_{b_1})$ (where $f/E_{b_1}$ is the tuple of appropriate classes) by Lemma \ref{lem:interpretation}, so we are done by the last point.
\end{proof} 

Let $T$ be the model completion of $T_\forall$. Then we have:

\begin{prop} \label{prop: models of T on open cones} 
    Suppose that $M \models T^*$. Then for any $c \in M$, $M_c \models T$ (see Lemma \ref{lem:interpretation}). 

    In fact, letting $T_{oc}$ say that the set of open cones of any center is a model of $T$, we have that $T^* = DT + T_{oc} + T_0$ where $T_0$ is the following axiom scheme: 

    $\bullet$ For all $x<y$ and for each isomorphism type of a model $A$ of $T_\forall$ such that $|A| = 1$, there is some $x<z<y$ such that $\{y/E_{z}\} \cong A$. (By this we mean that the open cone of $y$ of center $z$, with the induced $\Ll$-structure, is isomorphic to $A$.)
\end{prop}
\begin{proof}
    It is enough to prove the first part for the countable model $M$, which is ultrahomogeneous. By Lemma \ref{lem:interpretation} and Remark \ref{rem:weakly homogeneous} it follows that $M_c$ is also ultrahomogeneous and that the age of $M_c$ is precisely the class of finite $T_{\forall}$-models which implies that $M_c$ is the unique countable model of $T$ by Fact \ref{fac:building}.

    For the second part, note that $T^* \vdash DT$ (By \cite[Lemma 2.12]{Ramsey} or by checking the axioms given in Fact \ref{fac:dt}). It is also easy to see that $T^* \vdash T_0$. The equality follows from a straightforward back-and-forth using Lemma \ref{lem:types in T*} showing that $DT +T_{oc}+ T_0$ is $\omega$-categorical. 
    To see why we need $T_0$, consider the following situation. We have constructed some partial isomorphism $f:A \to B$ where $A \subseteq M$, $B \subseteq N$ are finite and $M,N \models T^*$, and we want to add one point $a \in M$ to the domain of $f$. Suppose that $a''< a < a'$ for some $a',a'' \in A$ which are minimal and maximal respectively, and let $b'=f(a')$, $b''=f(a'')$. For any point $b''<b < b'$, any $c\leq b''$ and any relation $R \in \Ll$, $N \models R^*(c,\ldots,b,\ldots)$ iff $N \models R^*(c,\ldots,b',\ldots)$, and the same is true in $M$. Thus, to extend $f$, we need to find some $b''<b<b'$ such that $M \models R^*(a,e)$ iff $N \models R^*(b,f(e))$ for any tuple of the appropriate length $a<e \in A$. Note that $M \models R^*(a,e)$ iff $M\models R^*(a,(a')^n)$. Let $C$ be the $\Ll$-structure induced by $M_a$ on $\{a'/E_{a}\}$. By $T_0$ we can find some $b'' < b <b'$ such that for all $R \in \Ll$, $N \models R^*(b,(b')^n)$ iff $C \models R((a'/E_{a})^n)$ iff $M \models R^*(a,(a')^n)$, so we are done. 
\end{proof}

\begin{cor} \label{cor:T* NIP} $T^*$ is NIP iff $T$ is NIP. Moreover, when $T$ is NIP, $T^*$ has finite dp-rank.\end{cor}
\begin{proof}
    If $T$ has IP, then as $T^*$ interprets $T$ with parameters (by Proposition \ref{prop: models of T on open cones}), it follows that $T^*$ has IP. 

    Suppose that $T$ is NIP. By the Sauer-Shelah lemma \cite[Lemma 6.4]{Sim}, it is enough to see that (*) there is some $k<\omega$ such that for any finite set $A \subseteq M \models T^*$ there are at most $|A|^k$ 1-types over $A$. By Remark \ref{rem:treef}, we may assume that $A$ is a substructure and by quantifier elimination it is enough to count quantifier-free types. Given $b$, let $b_1 = \max \Set{a \wedge b}{a \in A}$. By Lemma \ref{lem:types in T*}, it is enough to bound the number of possible quantifier-free types of $b/{E_{b_1}}$ over $\Set{a/E_{b_1}}{a\in A, a>b_1}$ in the structure $M_{b_1}$ (the fourth bullet there adds at most some bounded (in terms of $|\Ll|$) number of types). But since $T$ is NIP and $\Ll$ is finite, the analog of (*) for $T$ holds by Sauer-Shelah. 

    For the ``moreover'' part, note that the previous paragraph gives much more than just NIP: to show NIP it is enough to count $\varphi$-types over finite sets for any $\varphi$, and here we counted the full type. Suppose that the number of 1-types over $A$ is bounded by $|A|^k$. But then it cannot be that there is an ict-pattern (see Remark \ref{rem:ict pattern}) of depth $k+1$, so the dp-rank is $< k+1$. 
\end{proof}

We will mainly work with the following two examples (note that both satisfy the conditions listed in the beginning of this section).
\begin{exa} \label{exa:our examples}
    When $T_\forall$ is the theory of graphs in the language $\{R\}$ of graphs (so that $T$ is the theory of random graphs), denote $T^*$ by DTR. 

    When $T_\forall$ is the theory of two equivalence relations in the language $\{E_1,E_2\}$ (so that $T$ is the theory of two independent equivalence relations with infinitely many classes), denote $T^*$ by DTE$_2$. 
\end{exa}

\begin{cor} \label{cor:DTR has IP while TRE2 is NIP} DTR has IP and DTE$_2$ is NIP.
\end{cor}

The following is a bit technical, but it will be very useful later. Work in a monster model $\Cc \models T^*$. 

\begin{lem} \label{lem:indiscernible in tr -> in T*} Suppose that $I$ is an infinite $\emptyset$-indiscernible of single elements which is increasing or decreasing and without a maximal element (in the tree order), and that $I$ is $A$-indiscernible in the tree language $\{\leq, \wedge\}$. Then $I$ is $A$-indiscernible (in $\Ll^*$). \end{lem}
\begin{proof}
    Note that without the assumption of having no maximal element in $I$, the result might not be true. For example, suppose that $P(x)$ is some unary predicate in $\Ll$, and that $I = (a_i)_{i<\omega}$ is a decreasing sequence strictly bounded by some $c$. Then $I$ is $c$-indiscernible in the tree language, but it could be that $P^*(a_0,c)$ has a different truth value than $P^*(a_i,c)$ for $0<i$. 
    
    By maybe reversing the order and taking an infinite increasing subsequence (which exists as $I$ has no maximum), we may assume that $I$ is increasing and has order-type $\omega$ and write $I = (a_i)_{i<\omega}$. Note that as $I$ is $A$-indiscernible in the tree language, it follows that for any $d \in A$, either $d>I$, $d<a_0$ or $d \wedge a_i$ is constant and $= d \wedge a_0 <a_0$ for all $i<\omega$. We may assume that $A$ is a substructure. Let $D =  A\cup \Set{d \wedge a_0}{d\in A,d\wedge a_0 <a_0}$. Then $D$ is a substructure, $I$ is still indiscernible over $D$, and in addition, for any $i<\omega$ and any $d \in D$, either $a_i \wedge d = a_i$ or $a_i \wedge d \in D$.

    It follows that for any term $t(x,y)$ and any  increasing tuple $e \in I$ and any tuple $d \in A$, $t(e,d)$ is either $\in e$, or it $\in D$. In the first case, by indiscernibility, it is always the $k$'th element in the tuple $e$ (regardless of which tuple $e$ we take), which we denote by $e_k$, or it is constant.

    To show indiscernibility, by quantifier elimination it is enough to consider formulas of the form $\varphi(x,y) = R^*(t_0(x,y),t_1(x,y),\dots,t_n(x,y))$ when $R \in \Ll$ is some $n$-place relation and $t_0,\dots,t_n$ are terms. We show that for any increasing tuple $e$ from $I$ and any tuple $d\in D$, $\varphi(e,d)$ has the same truth value. By the previous paragraph, for any such $e$ and $d$, $t_i(e,d) = e_{k_i}$ for some constant $k_i < |e|$ or it is constantly $d_i \in D$ for any $i<n+1$.
    
    Let $s \subseteq n+1$ be the set of indices where $t_i(e,d) \in e$ (it does not depend on $e$). Divide into cases.

    \underline{Case I}: $s = \emptyset$. Then there is nothing to do.

    \underline{Case II}: $0 \in s$. If for some $i \notin s$, $d_i \not > I$, then by indiscernibility $d_i \not > e_{k_0}$ so by (\texttt{OC}), $\varphi(e,d)$ is always false and we are done. The same is true if for some $0<i \in s$, $e_{k_0}\geq e_{k_i}$. So we may assume that $d_i >I$ for all $i \notin s$ and that $e_{k_i} > e_{k_0}$ for all $0<i\in s$. By (\texttt{WD}), $R^*(e_{k_0},t_1(e,d),\dots,t_n(e,d))$ holds iff $R^*(e_{k_0},a_l,a_l,\dots,a_l)$ holds for any large enough $l$ (larger than the index of $e_{k_0}$ in $I$). This is because $t_i(e,d) \wedge a_l > e_{k_0}$ for all $0<i<n+1$. Note that such an $l$ exists since $I$ has no maximal element. Finally, if $e'$ is some other increasing tuple and $l$ is perhaps even larger, then the truth value of $R^*(e'_{k_0},a_l,a_l,\dots,a_l)$ is the same since $I$ is $\Ll^*$-indiscernible (over $\emptyset$).

    \underline{Case III}: $0 \notin s$. Again we may assume that $d_0 <I$ as otherwise by (\texttt{OC}), $\varphi(e,d)$ is always false. For simplicity, assume that $s = \Set{0<i<n+1}{i<m}$ for some $m\leq n+1$. Choose $l$ greater than all indices appearing in $e$. Then by (\texttt{WD}), (*) $R^*(d_0,e_{k_1},\dots,e_{k_m},d_{m+1},\dots,d_n)$ holds iff $R^*(d_0,a_l,\dots,a_l,d_{m+1},\dots,d_n)$ holds. If $e'$ is another increasing tuple, let $l$ be even larger than the indices appearing in $e'$, so (*) also holds for $e'$, and thus they have the same truth value. 
\end{proof}

\subsection{The type \texorpdfstring{$\pi_T$}{pi(T)}} \label{sec: the type}

Suppose that $T_{\forall}$, $T_{\forall}^*$, $T$ and $T^*$ are as above. 

Let $M\models T^*$ be a model with $\mathfrak{C}$ the monster, and $B$ a branch of $M$ (a maximal chain), consider the partial type $\rho_T(x)=\Set{x>b}{b\in B}$. 

\begin{prop}\label{prop:rho determines a complete type} $\rho_T$ determines a unique complete type $q_T(x)$ over $M$.\end{prop}
\begin{proof} 
    For every $m \in M \setminus{B}$, there is some $b \in B$ such that $b \not < m$, and it follows that $\rho\vdash x \wedge m = b \wedge m$. This implies that $\rho$ already determines a complete type over $M$ in the language of trees $\{\leq, \wedge\}$ and that for every $d\models \rho$, $Md$ is a substructure. 

    Hence, by QE it is enough to determine formulas of the form $R^*(c,-)$ for $c \in Md$. If $c = d$, $R^*(c,-)$ is empty in $Md$ by (\texttt{OC}). If $c \in M$, then for any tuple $c<f \in Md$, if $f$ does not contain $d$ there is nothing to do and otherwise $R^*(c,f)$ holds iff $R^*(c,f')$ holds where $f'$ is like $f$ with every occurrence of $d$ replaced by $b$ for some $c<b \in B$ (which must exist by choice of $B$).
\end{proof}

\begin{prop} \label{prop:dpr(q_T) >= dpr(x=x)} The dp-rank of $q_T$ is at least $\dpr(T)$ (the dp-rank of the partial type $x=x$ in $T$). In particular, if $T$ has IP, so does $q_T$ (see Remark \ref{rem:dp-rank infinite -> IP}).
\end{prop}
\begin{proof}
    We use the equivalent definition of dp-rank using ict-patterns, see Remark \ref{rem:ict pattern}. 

    Assume that $\dpr(T) \geq \kappa$. Let $c \models q_T$. In the $\Ll$-structure $\Cc_c = \Set{a/E}{a>c}$ which is a (saturated) model of $T$ by Proposition \ref{prop: models of T on open cones}, we can find an ict-pattern consisting of quantifier-free formulas $(\varphi_\alpha(x,y_\alpha))_{\alpha<\kappa}$ ($x$ a single variable) and an array of tuples $(a_{i,\alpha})_{i<\omega,\alpha<\kappa}$ witnessing this (by quantifier elimination in $T$). 

    Choose some tuples $d_{i,\alpha}$ such that $c < d_{i,\alpha}$ and $d_{i,\alpha}/E_c = a_{i,\alpha}$. For each $\eta:\kappa \to \omega$, there is some $b_\eta \in \Cc_c$ (an element) such that $\Cc_c\models \varphi(b_\eta,a_{i,\alpha})$ iff $\eta(\alpha) = i$. Let $e_\eta >c$ be such that $e_\eta/E_c = b_\eta$. Then by Lemma \ref{lem:interpretation}, we have that $\Cc\models \varphi_\alpha^*(c,e_\eta,d_{i,\alpha})$ iff $\eta(\alpha) = i$. Since $c < e_\eta$, $e_\eta \models \rho$, so $e_\eta \models q_T$ by Proposition \ref{prop:rho determines a complete type}, so we are done. 
\end{proof}

Let $c\models q_T$ in $\mathfrak{C}$ and $\pi_T(x)=q_T(x)\cup\{x<c\}$. 

\begin{prop} $\pi_T(x)$ does not fork over $M$ and moreover it is finitely satisfiable in $M$ (in fact in $B$).\end{prop}

\begin{proof} 
    This follows from Proposition \ref{prop:rho determines a complete type}, as $\pi_T(x)$ is equivalent to $\rho(x)\cup \{x<c\}$. 
\end{proof}

Next we would like to show that $\pi_T$ is dp-minimal and distal, so we will discuss distal types.

Distal theories were introduced by Simon in \cite{Distal} with the aim of capturing the ``purely unstable'' NIP theories. In \cite[Definition 2.4]{Travis}, Nell defined the notion of a ``distal type''. In addition one can define the dual notion. We give this definition now.

\begin{defi}\label{def:distal and co-distal}
    A partial type $\Sigma(x)$ is \emph{distal} if for every tuple $d$, and every indiscernible sequence $I_1+a+I_2$ of realizations of $\Sigma$, where $I_1$ and $I_2$ are infinite without endpoints, $I_1+I_2$ is $d$-indiscernible iff $I_1+a+I_2$ is $d$-indiscernible. 

    A partial type $\Sigma(x)$ over $A$ is \emph{co-distal over $A$} if for every $d \models \Sigma$ and every $A$-indiscernible sequence $I_1 + a + I_2$, if $I_1 + I_2$ is indiscernible over $Ad$, then $I_1 + a + I_2$ is indiscernible over $Ad$. If we do not specify $A$ we mean that $\Sigma$ is co-distal over its domain. 
\end{defi}

\begin{rk}
    Note that in the definition of a co-distal type, both the type and its domain play a role. If we extend either one, then co-distality is preserved. Indeed, suppose that $\Sigma(x)$ is co-distal over $A$, $A'\supseteq A$ and $\Sigma'(x)$ is a partial type over $A'$ extending $\Sigma$.
    Suppose that $d \models \Sigma'$ and that $I = I_1 +a +I_2$ is $A'$-indiscernible, $J=I_1 +I_2$ is $A'd$-indiscernible. Enlarging the tuples in $I$, we can assume that $A'$ is contained in each element from $I$. Since $I$ is $A$-indiscernible and $J$ is $Ad$-indiscernible, it follows by co-distality of $\Sigma$ that $I$ is $Ad$-indiscernible. But then it is $A'd$-indiscernible as well. 
\end{rk}

\begin{rk}
    If $p(x) \in S(A)$ is co-distal over $A$, then it is \emph{compressible} in the sense of \cite{1604.03841}: for every formula $\varphi(x,y)$ there is a formula $\theta(x,z)$ such that for any finite set $A_0 \subseteq A$ there is some $d \in A$ such that $\theta(x,d) \in p$ and $\theta(x,d) \vdash p \restriction A_0$. This easily follows from \textit{e.g.} the proof that strong honest definitions exist in distal theories in \cite[Theorem 9.21]{Sim}. 
\end{rk}

\begin{rk}
    A theory is distal iff every partial type is iff $x=x$ is for $x$ a single variable (see \cite[Theorem 2.28]{Distal}). 

    Similarly, a theory is distal iff every partial type is co-distal iff $x=x$ is for $x$ a single variable. 
\end{rk}

\begin{prop} \label{pro:distal implies NIP} Every distal or co-distal partial type is NIP.
\end{prop}
\begin{proof}
    We start with the distal case. Towards contradiction, suppose that $\Sigma(x)$ is a distal partial type over $A$ which has IP. Let $(a_i)_{i<\omega}$ be an $A$-indiscernible sequence of realizations of $\Sigma$, $\varphi(x,y)$ a formula and $b$ such that $\models\varphi(a_i,b)^{(\,i\,\mbox{even})}$ for all $i < \omega$. 
    By Ramsey and compactness, we may assume that $(a_{2i}a_{2i+1})_{i<\omega}$ is $Ab$-indiscernible. By compactness we can find such a sequence of any order type in which every element has a successor, which easily allows us to contradict distality. 

    For the co-distal case, if $\Sigma(x)$ is co-distal over $A$, let $(a_i)_{i<\omega}$ be an $A$-indiscernible sequence and $b$ any realization of $\Sigma$ witnessing IP and continue as above. 
\end{proof}

\begin{rk}
    One may ask what is the connection between distal and co-distal types. In the case of * being NIP or stable, co-* and * are equivalent, see Fact \ref{fac:pstable} and Lemma \ref{lem:pnip}. In general, we do not know, but it seems that by the proof of \cite[Theorem 2.28]{Distal}, distality and co-distality are equivalent under NIP, and also that co-distality always implies distality. We will not use this. 
\end{rk}

Finally, we have:

\begin{theo} \label{the:pi is dp-minimal and distal} The partial type $\{x<c\}$ is dp-minimal, distal and co-distal. In particular, it is NIP. 

As $\pi_T$ contains $x<c$, it follows that $\pi_T$ is dp-minimal, distal and co-distal.
\end{theo}

\begin{proof}
    We start by showing that $\{x<c\}$ is distal. So suppose that $I_1$, $I_2$, $a$ and $d$ are as in the definition: $I_1 + a + I_2$ is an indiscernible sequence of realizations of $\{x<c\}$, $I_1$ and $I_2$ are infinite without endpoints and $I_1 + I_2$ is $d$-indiscernible, and we have to show that $I_1 + a + I_2$ is $d$-indiscernible. 

    As $I = I_1 + a + I_2$ is bounded by $c$, it is either increasing or decreasing. Also, $J = I_1 + I_2$ is indiscernible over $d$, and in particular it is indiscernible over $d$ in the tree language. Thus, for any $e \in d$, either $e>J$, $e<J$ or $e \wedge b$ is constant and $<J$ for all $b\in J$. It follows that the same is true for $I$ and hence $I$ is indiscernible over $d$ in the tree language. Since $I$ is $\Ll$-indiscernible (over $\emptyset$) and has no last element, it follows from Lemma \ref{lem:indiscernible in tr -> in T*} that $I$ is $d$-indiscernible as required.

    Observe that in general, if $I$ is an increasing or decreasing sequence of elements, $I$ is indiscernible in the tree language over a set $A$ iff for any $c,d\in I$, $c \equiv_A d$. This observation will be useful to show co-distality. 

    We show that $\{x<c\}$ is co-distal. Suppose that $I = I_1 + a + I_2$ is $c$-indiscernible, $I_1,I_2$ are infinite without endpoints, $d<c$, $J=I_1 + I_2$ is $dc$-indiscernible but $I$ is not $dc$-indiscernible. By increasing the base $c$ to some finite set $A$ and replacing $I_1$ with an end-segment and $I_2$ with an initial segment we may assume that for some formula $\varphi(x,y)$ over $A$, $\varphi(d,b)$ holds for all $b\in J$ while $\models \neg \varphi(d,a)$. By compactness, we can find an $A$-indiscernible sequence $I'=(a_i)_{i\in \mathbb{Z}}$ such that $J' = (a_i)_{i\in \mathbb{Z}\setminus\{0\}}$ is $Ad$-indiscernible and for all $i \in \mathbb{Z}$, $\models \varphi(d,a_i)$ iff $i \neq 0$. 
    
    By applying $A$-automorphisms, we can find $(d_i)_{i\in \mathbb{Z}}$ such that $\tp(d_i a_j/A) = \tp(d a_0/A)$ for $j=i$ and $\tp(d_i a_j/A) = \tp(d a_1/A)$ for $i \neq j$ from $\mathbb{Z}$ (and these types are different as witnessed by $\varphi(x,y)$). Extracting, we may assume that $(a_i d_i)_{i\in \mathbb{Z}}$ is $A$-indiscernible. In particular, $I''=(d_i)_{i\in \mathbb{Z}}$ is an $A$-indiscernible sequence of elements below $c$, and in particular it is either increasing or decreasing without a maximal element (note that it cannot be constant). Let $J''=(d_i)_{i\in \mathbb{Z}\setminus\{0\}}$. Note that for $i \in \mathbb{Z}\setminus\{0\}$, $d_i \equiv_{Aa_0} d_j$, and thus by the observation above, $J''$ is indiscernible over $Aa_0$ in the tree language. As in the distal case, this implies that $I''$ is $Aa_0$-indiscernible in the tree language, and by Lemma \ref{lem:indiscernible in tr -> in T*} we get a contradiction.  

    Finally we show that $\{x<c\}$ is dp-minimal. By Fact \ref{fac:dp rank witnessing}, we can find some $A$ containing $c$ (such that $|A|$ is finite), two $A$-mutually indiscernible infinite sequences $I$ and $J$ such that for every element $b$ from $I$ or $J$, $b <c$,  and some $d<c$ such that neither $I$ nor $J$ are $Ad$-indiscernible. 

    By compactness we may assume that $I,J$ are both ordered by $\mathbb{Z}$. 
    As DT is dp-minimal (Fact \ref{fac:dt}), it follows that at least one of these sequences, say $I$, is indiscernible over $d$ in the tree language. Also, as $I$ is bounded by $c$, it is either increasing or decreasing. As it is ordered by $\mathbb{Z}$, it has no maximal element. Thus, we may use Lemma \ref{lem:indiscernible in tr -> in T*} to conclude.
\end{proof}

Recall Example \ref{exa:our examples}.

\begin{cor} \label{cor:main corollary about examples} The following hold:
    \begin{enumerate}
        \item DTR gives a counterexample to Question \ref{question}.
        \item DTE$_2$ shows that the Morley sequence in Theorem \ref{the:preserving dp-rank} was necessary even if the theory is NIP.
    \end{enumerate}
\end{cor}
\begin{proof}
    Both follow immediately from Proposition \ref{prop:dpr(q_T) >= dpr(x=x)}, Corollary \ref{cor:T* NIP} and Theorem \ref{the:pi is dp-minimal and distal}. For (1) we use the fact that the random graph has IP and for (2) that the theory EQ$_2$ of two independent equivalence relations is stable (thus NIP) of dp-rank at least 2 (actually exactly 2). (Corollary \ref{cor:T* NIP} was only used for (2)).
\end{proof}

\begin{rk} Note that Question \ref{question} was phrased for complete types, but by taking a non-forking or even a coheir extension of $\pi_{RG}$ we can find such a type. 
    
Also, note that as we cannot hope that our counterexample is definable over a model by Corollary \ref{cor:positive2}; the next best thing is a coheir, which is what we have. \end{rk}


 
\subsection{Classifying DTR} \label{sec:classifying DTR}

Let us summarize what we have so far. We know that Question \ref{question} has an affirmative answer for (generically) stable types or more generally when the Morley sequence generated by the type is totally indiscernible, when the base is a model and the type is definable over it, and when the theory is NIP or NSOP. DTR is clearly neither, but as we will see, it is NTP$_2$, and even inp-minimal.

To make the presentation a bit easier, we give a precise axiomatization of DTR in the language $\Ll_{TR} = \{\leq, \wedge, R\}$ where $R$ is a ternary relation symbol (this stood for $R^*$ in the previous section where $R$ was the edge relation). DTR is the model completion of the the theory TR (which is the theory $T^*_\forall$ for the theory of graphs), whose axioms are:
\begin{itemize}
\item[(T1)] \quad $\forall x \left( x\leq x \right)$
\item[(T2)] \quad $\forall x,y  \left( x\leq y \mathop{\&} y\leq x \longrightarrow  x=y \right)$
\item[(T3)] \quad $\forall x,y,z  \left( x\leq y \mathop{\&} y\leq z \longrightarrow x\leq z \right)$
\item[(T4)] \quad $\forall x,y,z  \left( y,z\leq x \longrightarrow y\leq z \lor z\leq y) \right)$
\item[(T5)] \quad $\forall x,y,z  \left( x\wedge y\leq x,y \right)$
\item[(T6)] \quad $\forall x,y,z  \left( z\leq x,y \longrightarrow z\leq x\wedge y \right)$
\item[(R1)] \quad $\forall x,y,z  \left( x<y\wedge z \longrightarrow \lnot R(x,y,z) \right)$
\item[(R2)] \quad $\forall x,y,z \left( R(x,y,z) \longrightarrow R(x,z,y) \right)$
\item[(R3)] \quad $\forall x,y,y',z \left( R(x,y,z) \mathop{\&} x<y\wedge y' \longrightarrow R(x,y',z) \right)$
\item[(R4)] \quad $\forall x,y,z \left( R(x,y,z) \longrightarrow x< y, z \right)$
\end{itemize}

To relate this axiomatization to the previous section, note that (R1) comes from irreflexivity, (R2) from symmetry, (R3) from (\texttt{WD}), and (R4) from (\texttt{OC}).

\begin{rk}\label{rem:rdos} Let $R'$ be the relation defined in our theory as, $R'(a,b)$ if and only if $a\land b\neq a, b$ and $R(a\land b,a,b)$. Note that we can replace $R$ with $R'$ and get the same definable sets, while still having quantifier elimination, because $R$ is quantifier-free definable from $R'$ and vice-versa. \end{rk}

 Recall that given a partial type $\Sigma(x)$ over $A$ and a cardinal $\kappa$, the \emph{burden} of $\Sigma$ is less than $\kappa$, denoted by $\bdn(\Sigma)<\kappa$, if for every $A$-mutually indiscernible sequences $(\bar{a}_{\alpha})_{\alpha<\kappa}$, with $\bar{a}_{\alpha}=(a_{\alpha,i})_{i<\omega}$ and every $b\models\Sigma$ there is $\beta<\kappa$ such that there exists $\bar{a}'$ indiscernible over $bA$ with $\bar{a}'\equiv_{a_{\beta,0}A}\bar{a}_{\beta}$. We write $\bdn(\Sigma)\geq\kappa$ if the negation holds. A theory $T$ is NTP$_2$ if and only if $\bdn(T)=\bdn(x=x)<\infty$ (here $x$ is a singleton), \textit{i.e.}, there is some cardinal $\kappa$ such that $\bdn(T)<\kappa$. If $\bdn(T)<2$ we  say that $T$ is \emph{inp-minimal}.
 (This is not the original definition of burden or NTP$_2$, but see \cite[Lemma 2.4 and Lemma 3.2]{Ch}.)

\begin{prop} \label{prop:DTR is inp-minimal} DTR is inp-minimal.\end{prop}
\begin{proof} We show $\bdn(DTR)<2$, as usual working in a monster model $\mathfrak{C}$. For this, let $I_0 = (a_{0,i})_{i\in\mathbb{Z}}, I_1=(a_{1,i})_{i\in\mathbb{Z}}$ be mutually indiscernible sequences and $c\in\mathfrak{C}$ an element in the monster model. We must show that for some $i<2$, there is some $I_i' \equiv_{ca_{i,0}} I_i$ which is indiscernible over $c$ (changing the order type does not matter).

Since DT is dp-minimal (Fact \ref{fac:dt}), one of the sequences must be $c$-indiscernible in the language of trees $\{\leq, \wedge\}$. Assume it is $I_0$ and to ease notation, let $I = I_0 = (a_i)_{i\in\mathbb{Z}}$. We are going to show that there is $I'$ indiscernible over $c$ in such that $I'\equiv_{a_0} I$.
 
Let $N=\langle Ic \rangle$ be the tree generated by $Ic$, and let $B = \langle I \rangle$. We are going to define $R^N$ in such a way that in $\mathcal{L}_{TR}$, (1) $N$ will be a model of TR, (2) $\langle a_0c \rangle \subseteq N$, (3) $a_ic \equiv^{\qf} a_0c$ in $N$ (\textit{i.e} having the same quantifier-free type) and (4) $B \subseteq N$. 

This will be enough, since then (as DTR is the model companion of TR) we can embed $N$ in $\mathfrak{C}$ fixing $a_0c$, thus getting $I'' = (a_i'')_{i\in \mathbb{Z}} \subseteq \mathfrak{C}$ with $a_0'' = a_0$ such that $a_i''c \equiv a_0c$ (by quantifier elimination) and $I'' \equiv_{a_0} I$. By Ramsey and compactness, and applying an automorphism, we can find a $c$-indiscernible sequence $I'$ as required. 

Let $A_0 = \langle a_0c \rangle$ (as an $\mathcal{L}_{TR}$-substructure of $\mathfrak{C}$). There are $\{\leq, \wedge\}$-isomorphisms $\sigma_i$ for $i \in \mathbb{Z}$ fixing $c$ such that $\sigma_i(a_0)=a_i$. For $i\in \mathbb{Z}$, let $A_i$ be the $\mathcal{L}_{TR}$-structures whose universe is $\sigma_i(A_0) = \langle a_ic \rangle$ and whose $\mathcal{L}_{TR}$-structure is the one induced by $\sigma_i$ (actually we only change the definition of $R$ on those structures). Clearly $A_i \models TR$ for all $i \in \mathbb{Z}$. 

Before the construction, recall the notation $E_c$ for ``being in the same open cone'', defined in the beginning of Section \ref{sec:trees}. 

We define $R^N$ as follows. We start with $S = \bigcup_{i \in \mathbb{Z}}{R^{A_i}} \cup R^{B}$ (as we must by (2), (3), (4) above) and we enlarge $S$ to satisfy (R3). Namely, for $b_0,b_1,b_2 \in N$ such that $b_0 = b_1 \wedge b_2 <b_1,b_2$ we let $R^N(b_0,b_1,b_2)$ iff for some $b_1',b_2'$ such that $b_1' \wedge b_1 >b_0$ and $b_2' \wedge b_2 >b_0$ (in other words, $b_1 \mathrel{E_{b_0}} b_1'$ and $b_2 \mathrel{E_{b_0}} b_2'$), $S(b_0,b_1',b_2')$. 

Now let us check (1), (2), (3) and (4) from above. 

(1) It is easy to see that $N \models TR$: (R1), (R2) and (R4) are clear, since these are clearly true with $S$ and the construction of $R$ does not harm this and (R3) is ensured by construction. 

(2) and (3) will both follow if we show that $A_i \subseteq N$ (as $\mathcal{L}_{TR}$-structures) for any $i \in \mathbb{Z}$. Fix some $i^* \in \mathbb{Z}$. We need to show that if $b_0,b_1,b_2 \in A_{i^*}$ and $N\models R(b_0,b_1,b_2)$ then $A_{i^*} \models R(b_0,b_1,b_2)$ (the other direction is by construction). By definition, this means that there are $b_1',b_2' \in N$ such that $b_0< b_1' \wedge b_1, b_2 \wedge b_2'$ and $S(b_0,b_1',b_2')$ holds. There are three cases.

\underline{Case i}: $(b_0, b_1',b_2') \in R^{A_{i^*}}$. Then there is nothing to show since $A_{i^*} \models TR$. 

\underline{Case ii}: $(b_0, b_1',b_2') \in R^{A_j}$ for some $j \neq {i^*}$. Write $b_0 = t_0(a_{i^*},c) = t_0'(a_j,c)$, $b_1 = t_1(a_{i^*},c)$, $b_2 = t_2(a_{i^*},c)$, $b_1' = t_1'(a_j,c)$ and $b_2' = t_2'(a_j,c)$ where $t_0,t_1,t_2,t_0',t_1',t_2'$ are terms. 
By indiscernibility of $I$ over $c$ in the tree language, it follows that $b_0 = t_0(a_i,c) = t_0'(a_i,c)$ for all $i \in \mathbb{Z}$ (\textit{i.e.} it is constant). It follows that $I$ is indiscernible over $b_0c$. 
Again by indiscernibility and the previous sentence, since $t_1(a_{i^*},c) \mathrel{E_{b_0}} t_1'(a_j,c)$ it follows that $t_1(a_{i},c) \mathrel{E_{b_0}} t_1'(a_i,c)$ for all $i\in \mathbb{Z}$, and similarly $t_2(a_i,c) \mathrel{E_{b_0}} t_2'(a_i,c)$ for all $i\in \mathbb{Z}$. 
Finally, since $(b_0, b_1',b_2') \in R^{A_j}$, it follows by construction that $(t_0(a_{i^*},c),t_1'(a_{i^*},c),t_2'(a_{i^*},c)) \in R^{A_{i^*}}$, and as $A_{i^*} \models TR$, it follows from the previous discussion that $(b_0,b_1,b_2) \in R^{A_{i^*}}$, as required.

\underline{Case iii}: $(b_0, b_1',b_2') \in R^{B}$. We use the notation $d_{i,n} = (a_i)_{i-n < i <i+n}$ for $i\in \mathbb{Z},n \in \omega$. Let $t_0,t_1,t_2$ be as in the previous case. Fix some $n<\omega$ so that we can write $b_0 = t_0(a_{i^*},c) = t_0'(d_{i^*,n})$, $b_1' = t_1'(d_{i^*,n})$ and $b_2' = t_2'(d_{i^*,n})$ for some terms $t_0',t_1',t_2'$.
As $I$ is indiscernible in $\mathcal{L}_{TR}$ (over $\emptyset$), it follows that $R(t_0(d_{0,n}),t_1'(d_{0,n}),t_2'(d_{0,n}))$ (in $\mathfrak{C}$, so in $B$). As $I$ in indiscernible over $c$ in the tree language, it follows that $t_0(a_0,c) = t_0'(d_{0,n})$ and that $t_1(a_0,c) \mathrel{E_{t_0(a_0,c)}} t_1'(d_{0,n})$ and $t_2(a_0,c) \mathrel{E_{t_0(a_0,c)}} t_2'(d_{0,n})$. As $\mathfrak{C} \models TR$, it follows that $A_0 \models R(t_0(a_0,c),t_1(a_0,c),t_2(a_0,c))$. By construction of $A_{i^*}$ it then follows that $A_{i^*} \models R(t_0(a_{i^*},c),t_1(a_{i^*},c),t_2(a_{i^*},c))$ as required.

Finally we make sure that (4) holds, \textit{i.e.} that $B \subseteq N$. We need to show that if $b_0,b_1,b_2 \in B$ and $N\models R(b_0,b_1,b_2)$ then $B \models R(b_0,b_1,b_2)$ (the other direction is by construction). By definition, there are $b_1',b_2' \in N$ such that $b_0< b_1' \wedge b_1, b_2 \wedge b_2'$ and $S(b_0,b_1',b_2')$ holds. There are two cases.

\underline{Case I}: $(b_0, b_1',b_2') \in R^B$. Then this is clear, since $B \models TR$. 

\underline{Case II}: $(b_0, b_1',b_2') \in R^{A_i}$ for some $i \in \mathbb{Z}$. The argument is analog to the one in Case iii above. Using the same notation, we write $b_0 = t_0(d_{i,n}) = t_0'(a_i,c)$, $b_1 = t_1(d_{i,n})$, $b_2 = t_2(d_{i,n})$, $b_1' = t_1'(a_i,c)$, $b_2' = t_2'(a_i,c)$ for some $n<\omega$ and terms $t_0,t_1,t_2,t_0',t_1',t_2'$. 
By construction of $A_i$, we have that $A_0 \models R(t_0'(a_0,c),t_1'(a_0,c),t_2'(a_0,c))$ (so this is true in $\mathfrak{C}$). 
As $I$ is indiscernible over $c$ in the language of trees, we have that $t_0(d_{0,n}) = t_0'(a_0,c)$, $t_1(d_{0,n}) \mathrel{E_{t_0(d_{0,n})}} t_1'(a_0,c)$ and $t_2(d_{0,n}) \mathrel{E_{t_0(d_{0,n})}} t_2'(a_0,c)$. 
It follows that $B \models R(t_0(d_{0,n}),t_1(d_{0,n}),t_2(d_{0,n}))$ (since this is true in $\mathfrak{C}$). 
Finally, as $I$ is indiscernible over $\emptyset$ in $\mathcal{L}_{TR}$, we have that $B \models R(t_0(d_{i,n}),t_1(d_{i,n}),t_2(d_{i,n}))$ as required. 

This concludes the proof. \end{proof}

\subsection{Non-distal example} \label{sec:non-distal example}

One may ask now whether Question \ref{question} fails for distal types but holds for non-distal ones. This is not the case, as we explain now. 
\begin{prop} There is a counterexample to Question \ref{question} in which the type is non-distal and also non-co-distal. \end{prop}

\begin{proof}

Consider the language $\Ll_{TR} \cup \{E\}$ where $E$ is a binary relation, and let TR$_E$ be TR + ``$E$ is an equivalence relation''. It is easy to see that we still have amalgamation for the class of finitely generated structures (amalgamate the TR-structure and then $E$ independently), so we again get a model completion DTR$_E$. This is just adding a generic equivalence relation to TR. It is not hard to see that DTR$_E$ expands DTR and that it expands $Tr_E$ which is the theory of trees with a generic equivalence relation (as in Proposition \ref{prop: models of T on open cones}). 

Note that TR$_E$ is NIP. This follows by an easy type-counting argument as in the proof of Corollary \ref{cor:T* NIP}: by the Sauer-Shelah lemma \cite[Lemma 6.4]{Sim}, it is enough to see that there is some $k<\omega$ such that for any finite set $A \subseteq M \models TR_E$ there are at most $|A|^k$ 1-types over $A$. By Remark \ref{rem:treef}, we may assume that $A$ is a substructure and by quantifier elimination it is enough to count quantifier-free types. Given $b$, let $b' = \max \Set{a \wedge b}{a \in A}$ and choose some $a\in A$ with $b \leq a$. The type $\tp(b/A)$ is determined by knowing the order-type of $b'$ over $A_{\leq a}$, whether $b > b'$ or $b = b'$ and the $E$-type of $bb'$ over $A$. This is easily polynomial in $|A|$. 
By the same argument in the ``moreover'' part of the proof of Corollary \ref{cor:T* NIP}, it follows that $Tr_E$ has finite dp-rank, say $<k$. 

Choose a branch $B$ in the countable model $M$. Let $\rho_E(x) = \Set{x>b}{b\in B}\cup\Set{\lnot E(x,a)}{a\in M}$. Precisely as in Proposition \ref{prop:rho determines a complete type}, $\rho_E$ determines a complete type $q_E$ over $M$ and is finitely satisfiable in $B$ (since in any interval there are infinitely many $E$-classes). 

As before let $c \models q_E$ and let $\pi_E(x) = q_E(x) \cup \{x<c\}$. Then $\pi_E(x)$ is NIP, in fact of dp-rank at most $k$, by the same proof as in Proposition \ref{cor:main corollary about examples}: given $k+1$ sequences, mutually indiscernible over some $A$ whose elements are realizations of $\pi_E$, and some $b\models \pi_E$, one of them is indiscernible over $Ab$ in the language of $Tr_E$ (because $\dpr(Tr_E) < k$), and then the same proof of Lemma \ref{lem:indiscernible in tr -> in T*} goes through (we only have to consider formulas of the form $R(t_0(x,y),t_1(x,y),t_2(x,y))$). 

Also, $q_E$ has IP by the same argument as in Proposition \ref{prop:dpr(q_T) >= dpr(x=x)}: in $\Cc_c$ we can realize a random graph. 

Finally, consider a $Bc$-indiscernible sequence $I = I_1+a+I_2$ of realizations of $\pi_E(x)$ in pairwise distinct $E$-classes. Let $d \models \pi_E$ be above $I$ and $d \mathrel{E} a$. Then $I_1+I_2$ is $Bcd$-indiscernible (by quantifier elimination and the axioms, noting that for any term $t(x)$ over $Bcd$ and any tuple $f$ from $I$, $t(f)$ is either $\in Bcd$ or in $f$), but $I_1+a+I_2$ is not. So $\pi_E$ is neither distal nor co-distal. \end{proof} 

\section{Open questions}
\label{sec:open questions}
Note that in our examples, after adding just one realization of the type to the base, the restriction became NIP and even distal and dp-minimal. In light of Theorem \ref{the:result1}, one can now ask for more refined examples. 

\begin{que2}
    Is there a theory $T$, a global NIP type $p$ non-forking over a set $A$, such that for every (some) Morley sequence $I = (a_i)_{i<\omega}$ and for every $n<\omega$, $p \restriction Aa_{<n}$ has IP?
\end{que2}

Alternatively, one may ask a weaker question, of finding different theories or different types for each $n<\omega$. 

A similar question can be asked about Theorem \ref{the:preserving dp-rank}, namely how much of the Morley sequence is actually needed to preserve the dp-rank.


It could be an interesting question to find a class of theories for which Question \ref{question} has an affirmative answer. Trivially NIP theories are such, but also NSOP theories by Remark \ref{rem:nsop}. We may discard NTP$_2$ and even inp-minimal theories because our counterexample.

It may be then interesting to restrict our attention to another class of NTP$_2$ such as Resilient theories (see \cite{BC}). Thus we have the following natural question:
\begin{que2} Is DTR resilient?\end{que2}
This could also be answered directly if the conjecture about the coincidence of Resilient and NTP$_2$ theories were proved (see for instance \cite[Question 4.14]{BC}).

Another class of theories one may consider is that of Rosy theories.

\begin{rk} Dense trees are not Rosy.\end{rk}
\begin{proof} To see this, we can apply the characterization given in \cite{EaOn}. They prove that a theory is Rosy if and only if it has ordinal-valued equivalence relation rank (see \cite[Definition 5.1]{EaOn}). In trees, an infinite decreasing chain of open cones witnesses that it has unbounded equivalence relation rank, so it is not Rosy.\end{proof}

The following question was asked by Pierre Simon during a talk given by the second author.

\begin{que2} Is there any counterexample to Question \ref{question} which is Rosy?\end{que2}
Similarly, we may ask
\begin{que2} Is it true that if $T$ is Rosy (and maybe also NIP), $p$ a global type non-forking over $A$, $\dpr(p) = \dpr(p\restriction A)$?\end{que2}

As well as being natural questions, given the remark above, it can be motivated by the fact that the original paper starting this line of research \cite{HassonOnshuus} was in the context of Rosy theories, and Corollary 2.27 from there precisely answers the stable variant of Question \ref{question} in that context.


\bibliographystyle{alpha}
\bibliography{Preservation}

\begin{thebibliography}{{Ram}19}

\bibitem[ACP14]{Gen}
Hans Adler, Enrique Casanovas, and Anand Pillay.
\newblock Generic stability and stability.
\newblock {\em J. Symb. Log.}, 79(1):179--185, 2014.

\bibitem[BYC14]{BC}
Ita\"{\i} Ben~Yaacov and Artem Chernikov.
\newblock An independence theorem for {${\rm NTP}_2$} theories.
\newblock {\em J. Symb. Log.}, 79(1):135--153, 2014.

\bibitem[Cas11a]{BCNS11}
Enrique Casanovas.
\newblock More on {NIP} and related topics. {L}ecture {N}otes of the {M}odel
  {T}heory {S}eminar, {U}niversity of {B}arcelona.
\newblock http://www.ub.edu/modeltheory/documentos/nip2.pdf, 2011.

\bibitem[Cas11b]{cassanovasBook}
Enrique Casanovas.
\newblock {\em Simple theories and hyperimaginaries}, volume~39 of {\em Lecture
  Notes in Logic}.
\newblock Association for Symbolic Logic, Chicago, IL, 2011.

\bibitem[Cas14]{Cas14}
Enrique Casanovas.
\newblock Lascar strong types and forking in {NIP} theories.
\newblock In {\em Rend. Sem. Mat. Univ. Politec. Torino,}, pages 195--201,
  2014.

\bibitem[CH14]{HilsChernikov}
Artem Chernikov and Martin Hils.
\newblock Valued difference fields and {$\text{NTP}_2$}.
\newblock {\em Israel J. Math.}, 204(1):299--327, 2014.

\bibitem[Che14]{Ch}
Artem Chernikov.
\newblock Theories without the tree property of the second kind.
\newblock {\em Ann. Pure Appl. Logic}, 165(2):695--723, 2014.

\bibitem[CK77]{C-K}
C.~C. Chang and H.~J. Keisler.
\newblock {\em Model theory}.
\newblock North-Holland Publishing Co., Amsterdam-New York-Oxford, second
  edition, 1977.
\newblock Studies in Logic and the Foundations of Mathematics, 73.

\bibitem[CS13]{ArtemPierreI}
Artem Chernikov and Pierre Simon.
\newblock Externally definable sets and dependent pairs.
\newblock {\em Israel J. Math.}, 194(1):409--425, 2013.

\bibitem[EO07]{EaOn}
Clifton Ealy and Alf Onshuus.
\newblock Characterizing rosy theories.
\newblock {\em J. Symbolic Logic}, 72(3):919--940, 2007.

\bibitem[HO10]{HassonOnshuus}
Assaf Hasson and Alf Onshuus.
\newblock Stable types in rosy theories.
\newblock {\em J. Symbolic Logic}, 75(4):1211--1230, 2010.

\bibitem[Hod93]{Ho}
Wilfrid Hodges.
\newblock {\em Model theory}, volume~42 of {\em Encyclopedia of Mathematics and
  its Applications}.
\newblock Cambridge University Press, Cambridge, 1993.

\bibitem[HP11]{HP}
Ehud Hrushovski and Anand Pillay.
\newblock On {NIP} and invariant measures.
\newblock {\em J. Eur. Math. Soc. (JEMS)}, 13(4):1005--1061, 2011.

\bibitem[KOU13]{KOU}
Itay Kaplan, Alf Onshuus, and Alexander Usvyatsov.
\newblock Additivity of the dp-rank.
\newblock {\em Trans. Amer. Math. Soc.}, 365(11):5783--5804, 2013.

\bibitem[KS14]{KSim}
Itay Kaplan and Pierre Simon.
\newblock Witnessing dp-rank.
\newblock {\em Notre Dame J. Form. Log.}, 55(3):419--429, 2014.

\bibitem[KS19]{KaplanSimon}
Itay Kaplan and Pierre Simon.
\newblock Automorphism groups of finite topological rank.
\newblock {\em Trans. Amer. Math. Soc.}, 372(3):2011--2043, 2019.

\bibitem[Nel19]{Travis}
Travis Nell.
\newblock Distal and non-distal behavior in pairs.
\newblock {\em MLQ Math. Log. Q.}, 65(1):23--36, 2019.

\bibitem[PT11]{AnandPredrag}
Anand Pillay and Predrag Tanovi\'{c}.
\newblock Generic stability, regularity, and quasiminimality.
\newblock In {\em Models, logics, and higher-dimensional categories}, volume~53
  of {\em CRM Proc. Lecture Notes}, pages 189--211. Amer. Math. Soc.,
  Providence, RI, 2011.

\bibitem[{Ram}19]{Ramsey}
Nicholas {Ramsey}.
\newblock {Invariants Related to the Tree Property}.
\newblock {\em arXiv e-prints}, page arXiv:1511.06453, August 2019.

\bibitem[Sim11]{Sim11}
Pierre Simon.
\newblock On dp-minimal ordered structures.
\newblock {\em J. Symbolic Logic}, 76(2):448--460, 2011.

\bibitem[Sim13]{Distal}
Pierre Simon.
\newblock Distal and non-distal {NIP} theories.
\newblock {\em Ann. Pure Appl. Logic}, 164(3):294--318, 2013.

\bibitem[Sim14]{PierreDprank}
Pierre Simon.
\newblock Dp-minimality: invariant types and dp-rank.
\newblock {\em J. Symb. Log.}, 79(4):1025--1045, 2014.

\bibitem[Sim15]{Sim}
Pierre Simon.
\newblock {\em A guide to {NIP} theories}, volume~44 of {\em Lecture Notes in
  Logic}.
\newblock Association for Symbolic Logic, Chicago, IL; Cambridge Scientific
  Publishers, Cambridge, 2015.

\bibitem[Sim19]{1604.03841}
Pierre Simon.
\newblock Type decomposition in {NIP} theories.
\newblock {\em J. Eur. Math. Soc. (JEMS)}, 2019.

\bibitem[TZ12]{TZ}
Katrin Tent and Martin Ziegler.
\newblock {\em A course in model theory}, volume~40 of {\em Lecture Notes in
  Logic}.
\newblock Association for Symbolic Logic, La Jolla, CA; Cambridge University
  Press, Cambridge, 2012.

\end{thebibliography}

\vspace{10pt}
\noindent\textit{Pedro Andr\'es Estevan}\\
\noindent\textit{Departament de Matemàtiques i Informàtica}\\
\noindent\textit{Universitat de Barelona}\\
\textit{p.a.estevani@gmail.com}\\

\vspace{5pt}
\noindent\textit{Itay Kaplan}\\
\noindent\textit{Einstein Institute of Mathematics}\\
\textit{The Hebrew University of Jerusalem, 91904, Jerusalem Israel}\\
\textit{kaplan@math.huji.ac.il}\\

\end{document}